\documentclass[a4paper,12pt]{amsart}
\usepackage{amsmath,amsthm,amssymb}
\usepackage[latin1]{inputenc}
\usepackage[T1]{fontenc}
\usepackage[all]{xy}
\usepackage{mathrsfs}

\numberwithin{equation}{section}
\newtheorem{thm}[equation]{Theorem}
\newtheorem*{prop*}{Proposition}
\newtheorem*{propa*}{Proposition A}
\newtheorem*{propb*}{Proposition B}
\newtheorem*{thma*}{Theorem A}
\newtheorem*{thm*}{Theorem}
\newtheorem{cor}[subsection]{Corollary}
\newtheorem{lem}[subsection]{Lemma}
\newtheorem{prop}[equation]{Proposition}

\theoremstyle{definition}

\newtheorem{rem}[equation]{Remark}

\DeclareMathOperator{\h}{H}

\newcommand{\triv}{{\mathbf{1}}}
\def\R{\mathbb R}

\def\Z{\mathbb Z}
\def\A{\mathbb A}
\def\Q{\mathbb Q}
\def\C{\mathbb C}
\def\F{\mathbb F}
\def\E{\mathcal E}

\newcommand{\Whit}{\mathcal{W}}
\newcommand{\whit}{W}

\def\ira{\stackrel{\sim}{\longrightarrow}}
\def\hra{\hookrightarrow}

\def\ra{\rightarrow}

\def\g{\mathfrak g}

\def\t{\textbf{\emph{t}}}
\def\s{\mathfrak s}

\def\a{\mathfrak a}

\def\O{\mathcal O}
\def\h{\mathfrak h}
\def\c{\mathfrak c}
\def\m{\mathfrak m}
\def\k{\mathfrak k}

\def\<{\langle}
\def\>{\rangle}

\def\GL{{\rm GL}}
\def\SSS{\mathcal{S}}
\def\SO{{\rm SO}}

\def\Hom{{\rm Hom}}

\def\Whit{\mathcal{W}}
\def\z{\mathfrak{z}}

\newcommand{\sm}{\mathscr{C}^\infty}

\title[]{A rationality result for\\ the exterior and the symmetric square $L$-function.}
\author{Harald Grobner}
\date{\today}

\address{Harald Grobner: Fakult\"at f\"ur Mathematik, University of Vienna\\ Oskar--Morgenstern--Platz 1\\ A-1090 Vienna, Austria.}
\email{harald.grobner@univie.ac.at}
\urladdr{http://homepage.univie.ac.at/harald.grobner}

\thanks{H.G. is supported by the Austrian Science Fund (FWF), project number P 25974-N25.}

\begin{document}
\maketitle
%\small
\centerline{{\it With an appendix by Nadir Matringe}}
\normalsize
\begin{abstract}
Let $G=\GL_{2n}$ over a totally real number field $F$ and $n\geq 2$. Let $\Pi$ be a cuspidal automorphic representation of $G(\A)$, which is cohomological and a functorial lift from SO$(2n+1)$. The latter condition can be equivalently reformulated that the exterior square $L$-function of $\Pi$ has a pole at $s=1$. In this paper, we prove a rationality result for the residue of the exterior square $L$-function at $s=1$ and also for the holomorphic value of the symmetric square $L$-function at $s=1$ attached to $\Pi$. On the way, we also show a rationality result for the residue of the Rankin--Selberg $L$-function at $s=1$, which is very much in the spirit of our recent joint paper with Harris and Lapid \cite{grob_harris_lapid}, as well as of one of the main results in a recent article of Balasubramanyam--Raghuram \cite{bala_ragh}
\end{abstract}

\setcounter{tocdepth}{1}
\tableofcontents
\section{Introduction}
\subsection{General background}
Let $F$ be an algebraic number field and let $\Pi$ be a cuspidal automorphic representation of $\GL_2(\A_F)$. Rationality results for special values of the associate automorphic $L$-function $L(s,\Pi)$ have been studied by several authors over the last decades. For the scope of this paper, we would like to mention Y. Manin and G. Shimura, who were the first to study special values of $L(s,\Pi)$ in the particular case, when $F$ is totally real, i.e., when $\Pi$ comes from a Hilbert modular form, cf.\ \cite{manin} and \cite{shimura}, and P. F. Kurchanov, who treated the case of a CM-field $F$ in a series of (meanwhile seemingly forgotten) papers, cf.\ \cite{kurc1, kurc2}. Shortly later, G. Harder published some articles, see \cite{hardermodsym, harder_periodint}, in which he described a general approach to such rationality results. In \cite{hardermodsym}, Harder considered the case of an arbitrary number field $F$, while in \cite{harder_periodint}, he extended the methods of the above authors to some automorphic representations, which do not necessarily come from cusp forms (for $F$ imaginary quadratic). The case of $\GL_2$ over a general number field $F$ has also been considered later on by Shimura, see \cite{shimura_Invent}.

It took some time until extensions of these results to general linear groups $\GL_n$ of higher rank $n$ were available. Important achievements include Ash-Ginzburg, \cite{ash-ginzburg}, Kazhdan-Mazur-Schmidt, \cite{kazhdan-mazur-schmidt} and Mahnkopf, \cite{mahnk}.

Guided by the above methods, meanwhile, there is a growing number of results that have been proved about the rationality of special values of certain automorphic $L$-functions attached to $\GL_n$. As a selection of examples, relevant to the present paper, we refer to Raghuram \cite{raghuram-imrn, ragh-gln}, Harder-Raghuram \cite{harder-raghuram}, Grobner-Harris \cite{grob_harris}; Grobner-Raghuram \cite{grob-ragh}, Grobner-Harris-Lapid \cite{grob_harris_lapid} and Balasubramanyam-Raghuram \cite{bala_ragh}.

In all of these reference, the corresponding rationality result is obtained by writing the special $L$-value at hand as an algebraic multiple of a certain period invariant\footnote{The approach taken in \cite{grob_harris_lapid}, however, is a certain, basis-free variation of the latter.}. This period is defined by comparison of a rational structure on a cohomology space, attached to the given automorphic representation $\Pi$, with a rational structure on a model-space of (the finite part of) $\Pi$, such as a Whittaker model or a Shalika model. (The word ``rational structure'' here refers to a subspace of the vector space, carrying the action of $\Pi$, which is essentially defined over the field of rationality of $\Pi$ and at the same time stable under the group action.) While the first rational structure on the cohomology space is purely of geometric nature and has its origin in the cohomology of arithmetic groups (or better: the cohomology of arithmetically defined locally symmetric spaces), the latter rational structure is defined by reference to the uniqueness to the given model-space.

In this paper, we continue the above considerations. But while most of the aforementioned papers deal with special values of the Rankin--Selberg $L$-function (by some variation or the other), the principal $L$-function, or the Asai $L$-functions, here we would like to study the algebraicity of the exterior square $L$-function and the symmetric square $L$-function, attached to a cuspidal automorphic representation of the general linear group.

\subsection{The main results of this paper}
To put ourselves {\it in medias res}, let $F$ be a totally real number field and let $G=\GL_{N}/F$, $N=2n$ with $n\geq 2$. The restrictions on $F$ and the index of the general linear groups under consideration are owed to the inevitable, as it will become clear below. Indeed, let $\Pi$ be a cuspidal automorphic representation of $G(\A)$ and let $\Whit^\psi(\Pi)$ be its $\psi$-Whittaker model. As we want to exploit the results of Bump--Friedberg, \cite{bumbginzburg}, we shall assume that the partial exterior square $L$-function $L^S(s,\Pi,\Lambda^2)$ of $\Pi$ has a pole at $s=1$. (Here, $S$ is a finite set of places of $F$, containing all archimedean ones, such that for a place $v\notin S$, the local components $\Pi_v$ and $\psi_v$ are unramified.) In particular, this forces $N=2n$ to be even, see \cite{jac_shal_AnnArbor}, Thm.\ 2, and furthermore $\Pi$ to be self-dual, $\Pi\cong\Pi^\vee$, and to have trivial central character.

Our first main result gives a rationality statement for the residue ${\rm Res}_{s=1}(L^S(s,\Pi,\Lambda^2))$ of the exterior square $L$-function. More precisely, we obtain the following result:

\begin{thm}\label{thm:intro1}
Let $F$ be a totally real number field and $G=\GL_{2n}/F$, $n\geq 2$. Let $\Pi$ be a unitary cuspidal automorphic representation of $G(\A)$, which is cohomological with respect to an irreducible, self-contragredient, algebraic, finite-dimensional representation $E_\mu$ of $G_\infty$. Assume that $\Pi$ satisfies the equivalent conditions of Prop.\ \ref{prop:lift}, i.e., the partial exterior square $L$-function $L^S(s,\Pi,\Lambda^2)$ has a pole at $s=1$. Then, for every $\sigma\in {\rm Aut}(\C)$, there is a non-trivial period $p^t({}^\sigma\Pi)$, defined by a comparison of a given rational structure on the Whittaker model of ${}^\sigma\Pi_f$ and a rational structure on a realization of ${}^\sigma\Pi_f$ in cohomology in top degree $t$, and a non-trivial archimedean period $p^t({}^\sigma\Pi_\infty)$, such that
$$
\sigma\left(
\frac{L(\tfrac12,\Pi_f)\cdot {\rm Res}_{s=1}(L^S(s,\Pi,\Lambda^2))}{p^t(\Pi) \ p^t(\Pi_\infty)}\right) \ = \
\frac{L(\tfrac12,{}^\sigma\Pi_f)\cdot {\rm Res}_{s=1}(L^S(s,{}^\sigma\Pi,\Lambda^2))}{p^t({}^\sigma\Pi) \ p^t({}^\sigma\Pi_\infty)}.
$$
In particular,
$$
L(\tfrac12,\Pi_f)\cdot {\rm Res}_{s=1}(L^S(s,\Pi,\Lambda^2)) \ \sim_{\Q(\Pi_f)} \
p^t(\Pi) \ p^t(\Pi_\infty),
$$
where ``$\sim_{\Q(\Pi_f)}$'' means up to multiplication of the right hand side by an element in the number field $\Q(\Pi_f)$.
\end{thm}
This is proved in details in \S\ref{sect:residue}, see Thm.\ \ref{thm:residue}. For a precise definition of the periods $p^t({}^\sigma\Pi)$ and $p^t({}^\sigma\Pi_\infty)$ we refer to Prop.\ \ref{prop:periods}, respectively \eqref{eq:arch_period}, the non-vanishing of the archimedean period $p^t({}^\sigma\Pi_\infty)$ being shown -- building on a result of B. Sun -- in our Thm.\ \ref{hyp}. The number field $\Q(\Pi_f)$ in the theorem is (by Strong Multiplicity One) the aforementioned field of rationality of the cuspidal automorphic representation $\Pi$. See \S\ref{sect:sigmatwist} and \S\ref{sect:rat_struct}.

The key result, which we use, in order to derive the above theorem, is a certain integral-representation, obtained by Bump-Friedberg, \cite{bumbginzburg}, of the residue ${\rm Res}_{s=1}(L^S(s,\Pi,\Lambda^2))$ of the exterior square $L$-function in terms of integrating over a cycle $Z(\A)H(F)\backslash H(\A)$. Here, $Z$ is the centre of $G$ and $H=\GL_n\times\GL_n$, suitably embedded into $G$, cf.\ \ref{sect:GH}. 

More precisely, if one combines the three main results of \cite{bumbginzburg}, then, under the assumptions made in the theorem, one obtains the following equality, shown in our Thm.\ \ref{thm:integralrep_extsqr}:
\begin{equation}\label{eq:BG}
c_n\cdot\hat\Phi(0) \int_{Z(\A)H(F)\backslash H(\A)}\varphi(J(g,g')) \ dg \ dg' = \frac{L^S(\tfrac12,\Pi)\cdot {\rm Res}_{s=1}(L^S(s,\Pi,\Lambda^2))}{L^S(n,\triv)^2}\cdot \prod_{v\in S}\frac{Z_v(\xi_v,f_{v,1})}{L(n,\triv_v)}.
\end{equation}
Here, $\Phi$ is a certain global Schwartz-Bruhat function on $\A^n$, chosen with care in \S\ref{sect:SchwartzBruhat}, and $c_n\cdot\hat\Phi(0)$ is the (non-zero) residue at $s=1$ of an Eisenstein series attached to a section $f_s=\otimes_v f_{v,s}$, which is defined by $\Phi$. See \S\ref{sect:SchwartzBruhat} and \S\ref{sect:BG} for the precise definitions of the terms appearing in \eqref{eq:BG}. What one should observe is that the value of the partial $L$-function $L^S(n,\triv)$ of the trivial character of $\A$ at $n$ appears in the formula. In order for the pole of $L^S(s,\Pi,\Lambda^2)$ at $s=1$ not to cancel with the pole of $L^S(n,\triv)$ at $n=1$, we assumed $n\geq 2$, which explains the corresponding assumption in Thm.\ \ref{thm:intro1} (resp.\ Thm.\ \ref{thm:residue}). (As for the case of $n=1$, $L^S(s,\Pi,\Lambda^2)=L^S(s,\triv)$, the analogue of Thm.\ \ref{thm:intro1} would boil down to a rationality result for the central critical value of the $L$-function of unitary cusp forms of $\GL_2(\A)$, which is known, e.g., by Harder \cite{hardermodsym}. Therefore, considering only $n\geq 2$ is not a serious restriction.)

Observe that the top degree $t$, mentioned in Thm.\ \ref{thm:intro1}, where $\Pi$ has non-trivial cohomology, equals the dimension of the locally symmetric spaces, which are associate to the cycle $Z(\A)H(F)\backslash H(\A)$, cf.\ \S\ref{sect:deRham}. (Here, we necessarily have to use that $F$ is totally real, which explains the last obstruction, set in the beginning.) As a consequence, we may use the de Rham isomorphism. Together with \eqref{eq:BG} and N. Matringe's equivariance-result (Thm.\ A) in the appendix, this finally gives Thm.\ \ref{thm:intro1}.

We point out that, if $\Pi$ satisfies the assumptions made in the theorem, then $\Pi$ automatically satisfies the assumptions made in Grobner-Raghuram \cite{grob-ragh}. Hence, the non-zero periods $\omega^{\epsilon_0}(\Pi_f)$ and $\omega(\Pi_\infty)$ constructed in {\it loc.\ cit.} are well-defined. See our \S\ref{sect:WSext} below for details. If we define the non-zero, top-degree Whittaker-Shalika periods,
$$P^t(\Pi):=\frac{p^t(\Pi)}{\omega^{\epsilon_0}(\Pi_f)}\quad\quad\textrm{and}\quad\quad P^t(\Pi_\infty):=\frac{p^t(\Pi_\infty)}{\omega(\Pi_\infty)},$$
then we may get rid of the $L$-factor $L(\tfrac12,\Pi_f)$ in Thm.\ \ref{thm:intro1} above. The following result is Cor.\ \ref{cor:Wext}.

\begin{cor}
Let $\Pi$ be as in the statement of Thm.\ \ref{thm:intro1} (resp.\ Thm.\ \ref{thm:residue}). Then
$${\rm Res}_{s=1}(L^S(s,\Pi,\Lambda^2)) \ \approx_{\Q(\Pi_f)} \
P^t(\Pi) \ P^t(\Pi_\infty),$$ 
where ``$\approx_{\Q(\Pi_f)}$'' means up to multiplication of both sides by an element in the number field $\Q(\Pi_f)$.
\end{cor}

In order to obtain our second main theorem on the symmetric square $L$-function, we need to prove a version of one of the main results of Grobner-Harris-Lapid \cite{grob_harris_lapid} and Balasubramanyam-Raghuram \cite{bala_ragh}, which is tailored to our present situation at hand. This is done in \S\ref{sect:RS}. The main result of this section reads as follows

\begin{thm}\label{thm:intro2}
Let $F$ be a totally real number field and $G=\GL_{2n}/F$, $n\geq 2$. Let $\Pi$ be a self-dual, unitary, cuspidal automorphic representation of $G(\A)$ (with trivial central character), which is cohomological with respect to an irreducible, self-contragredient, algebraic, finite-dimensional representation $E_\mu$ of $G_\infty$. Then, for every $\sigma\in {\rm Aut}(\C)$,
$$
\sigma\left(
\frac{{\rm Res}_{s=1}(L^S(s,\Pi\times\Pi))}{p^t(\Pi) \ p^b(\Pi) \ p(\Pi_\infty)}\right) \ = \
\frac{{\rm Res}_{s=1}(L^S(s,{}^\sigma\Pi\times{}^\sigma\Pi))}{p^t({}^\sigma\Pi) \ p^b({}^\sigma\Pi) \ p({}^\sigma\Pi_\infty)}.
$$
In particular,
$$
{\rm Res}_{s=1}(L^S(s,\Pi\times\Pi)) \ \sim_{\Q(\Pi_f)^\times} \
p^t(\Pi) \ p^b(\Pi) \ p(\Pi_\infty),
$$
where ``$\sim_{\Q(\Pi_f)^\times}$'' means up to multiplication by a non-trivial element in the number field $\Q(\Pi_f)$.
\end{thm}
In the statement of the latter theorem, $p^t(\Pi)$ is the top-degree period defined above, while $p^b(\Pi)$ is defined analogously, but using the lowest degree $b$, where $\Pi$ has non-trivial cohomology. The non-vanishing archimedean period $p(\Pi_\infty)$ is defined in \eqref{eq:arch_period_RS}. We refer to \S\ref{sect:bottom_periods} and \S\ref{sect:archimedean_RS} for precise assertions and definitions concerning these periods. The proof of Thm.\ \ref{thm:intro2} is finally given in \S\ref{sect:Rankin_Selberg}, Thm.\ \ref{thm:Rankin_Selberg} and relies on the usual unfolding of the Rankin--Selberg integral, \S\ref{sect:integralrep_RS}.\\\\
The second main theorem of this paper finally deals with the value of the symmetric square $L$-function at $s=1$. Recall that we have
$$L^S(s,\Pi\times\Pi)=L^S(s,\Pi,{\rm Sym}^2)\cdot L^S(s,\Pi,\Lambda^2).$$
As by assumption $L^S(s,\Pi,\Lambda^2)$ carries to (simple) pole of $L^S(s,\Pi\times\Pi)$ at $s=1$, the symmetric square $L$-function $L^S(s,\Pi,{\rm Sym}^2)$ is holomorphic and non-vanishing at $s=1$. Our second main theorem hence follows by combining Thm.\ \ref{thm:intro1} (resp.\ Thm.\ \ref{thm:residue}) with Thm.\ \ref{thm:intro2} (resp.\ Thm.\ \ref{thm:Rankin_Selberg}). We obtain, see Thm.\ \ref{thm:Symm2},

\begin{thm}\label{thm:intro3}
Let $F$ be a totally real number field and $G=\GL_{2n}/F$, $n\geq 2$. Let $\Pi$ be a unitary cuspidal automorphic representation of $G(\A)$, which is cohomological with respect to an irreducible, self-contragredient, algebraic, finite-dimensional representation $E_\mu$ of $G_\infty$. Assume that $\Pi$ satisfies the equivalent conditions of Prop.\ \ref{prop:lift}, i.e., the partial exterior square $L$-function $L^S(s,\Pi,\Lambda^2)$ has a pole at $s=1$. Then, for every $\sigma\in {\rm Aut}(\C)$,
$$
\sigma\left(
\frac{L(\tfrac12,\Pi_f)\ p^b(\Pi) \ p^b(\Pi_\infty)}{L^S(1,\Pi,{\rm Sym}^2)}\right) \ = \
\frac{L(\tfrac12,{}^\sigma\Pi_f)\ p^b({}^\sigma\Pi) \ p^b({}^\sigma\Pi_\infty)}{L^S(1,{}^\sigma\Pi,{\rm Sym}^2)}.
$$
In particular,
$$
L^S(1,\Pi,{\rm Sym}^2) \sim_{\Q(\Pi_f)} L(\tfrac12,\Pi_f)\ p^b(\Pi) \ p^b(\Pi_\infty)
$$
where ``$\sim_{\Q(\Pi_f)}$'' means up to multiplication of $L^S(1,\Pi,{\rm Sym}^2)$ by an element in the number field $\Q(\Pi_f)$.
\end{thm}
Similar to before, we may define bottom-degree Whittaker-Shalika periods. Set
$$P^b(\Pi):=p^b(\Pi)\cdot\omega^{\epsilon_0}(\Pi_f)\quad\quad\textrm{and}\quad\quad P^b(\Pi_\infty):=p^b(\Pi_\infty)\cdot\omega(\Pi_\infty).$$
Then, we have the following corollary, see Cor.\ \ref{cor:Wsymm}, in which we get once more rid of the $L$-factor $L(\tfrac12,\Pi_f)$.

\begin{cor}
Let $\Pi$ be as in the statement of Thm.\ \ref{thm:intro3} (resp.\ Thm.\ \ref{thm:Symm2}). Then
$$L^S(1,\Pi,{\rm Sym}^2)\ \approx_{\Q(\Pi_f)} \ P^b(\Pi) \ P^b(\Pi_\infty),$$
where ``$\approx_{\Q(\Pi_f)}$'' means up to multiplication of both sides by an element in the number field $\Q(\Pi_f)$.
\end{cor}

\noindent{\it Acknowledgements: } \small
I am grateful to Erez Lapid for suggesting me to work in this direction.
\normalsize

\section{Notation and conventions}

\subsection{Number fields}
In this paper, $F$ denotes a totally real number field of degree $d=[F:\Q]$ with ring of integers $\O$. For a place $v$, let $F_v$ be the topological completion of $F$ at $v$. Let $S_\infty$ be the set of archimedean places of $F$. If $v \notin S_{\infty}$, we let $\O_v$ be the local ring of integers of $F_v$ with unique maximal ideal $\wp_v$. Moreover, $\A$ denotes the ring of ad\`eles of $F$ and as usually, $\A_f$ its finite part. We use the local and global normalized absolute values, the first being denoted by $|\cdot|_v$, the latter by $\|\cdot\|$. The fact that $F$ has no complex place is crucial, see \S\ref{sect:deRham}. Once and for all, we fix a non-trivial additive character $\psi: F\backslash\A\ra\C^\times$ (normalized where it is unramified).

\subsection{Algebraic groups an real Lie groups}\label{sect:GH}
Throughout this paper $G$ denotes $\GL_{2n}/F$, $n\geq 2$, the split general linear group over $F$. Although much of the paper works also for $\GL_N$ with $N$ arbitrary (e.g., the diagram \ref{diagram}), it will be crucial for the main result that $N=2n$ is even (because only then, the exterior square $L$-function may have a pole, \cite{jac_shal_AnnArbor}, Thm.\ 2, p.\ 224) and that $n\geq 2$ (because the $\zeta$-function attached to $F$ has a pole at $n$, if $n=1$). Let $H$ be $\GL_n\times \GL_n$ over $F$. We identify $H$ with a subgroup of $G$, defined as the image of the homomorphism $J:\GL_n\times \GL_n \ra \GL_{2n}$, where
$$J(g,g')_{k,l}:=\left\{
\begin{array}{ll}
g'_{i,j} & \textrm{if $k=2i-1$ and $l=2j-1$}\\
g_{i,j} & \textrm{if $k=2i$ and $l=2j$}\\
0 & \textrm{else}
\end{array}
\right.$$
The center of $G/F$ is denoted $Z/F$. We write $G_\infty:=R_{F/\Q}(G)(\R)$ (resp., $H_\infty:=R_{F/\Q}(H)(\R)$ or $Z_\infty:=R_{F/\Q}(Z)(\R)$), where $R_{F/\Q}$ stands for Weil's restriction of scalars. Lie algebras of real Lie groups are denoted by the same letter, but in lower case gothics.

At an archimedean place $v\in S_{\infty}$ we let $K_v$ be a maximal compact subgroup of the real Lie group $G(F_v)=\GL_{2n}(\R)$. It is isomorphic to $O(2n)$. We write $K_v^\circ$ for the connected component of the identity of $K_v$, which is isomorphic to $SO(2n)$. We set $K_\infty:=\prod_{v\in S_\infty} K_v$ and $K^\circ_\infty:=\prod_{v\in S_\infty} K^\circ_v$. Moreover, we denote by $K_{H,v}$ the intersection $K_v\cap H(F_v)$, which is a maximal compact subgroup of $H(F_v)$, isomorphic to $O(n)\times O(n)$. As before, we write $K^\circ_{H,v}$ for the connected component of the identity and we let $K_{H,\infty}:=\prod_{v\in S_\infty} K_{H,v}$ and $K^\circ_{H,\infty}:=\prod_{v\in S_\infty} K^\circ_{H,v}$.

Let $A_G$ be the multiplicative group of positive real numbers $\R_+$, being diagonally embedded into the center $Z_{\infty}$ of $G_\infty$. It is a direct complement of the group $G(\A)^{1}:=\{g\in G(\A)| \|\det(g)\|=1\}$. According to our conventions, the Lie algebra of the real Lie group $A_G$ is denoted $\a_G$. Furthermore, we let $\m_G:=\g_\infty/\a_G$, $\m_H:=\h_\infty/\a_G$ and $\s:=\z_\infty/\a_G$. Observe that these spaces are Lie subalgebras of $\g_\infty$.

\subsection{Coefficient modules}\label{sect:highweights}
In this paper, $E_\mu$ denotes an irreducible, algebraic representation of $G_{\infty}$ on a finite-dimensional
complex vector space. It is determined by its highest weight $\mu=(\mu_v)_{v\in S_\infty}$, whose local components at an archimedean place $v$ may be identified with $\mu_v=(\mu_{1,v},...,\mu_{2n,v})\in\Z^{2n}$, $\mu_{1,v}\geq\mu_{2,v}\geq...\geq\mu_{2n,v}$.
We assume that $E_\mu$ is self-dual, i.e., it is isomorphic to its contragredient, $E_\mu\cong E_\mu^\vee$, or, in other words, that
$$\mu_{j,v}+\mu_{2n-j+1,v}=0, \quad\quad 1\leq j\leq n$$
at all places $v\in S_\infty$. Clearly, this condition implies that $\mu_{n,v}\geq 0\geq \mu_{n+1,v}$ for all $v\in S_\infty$. The self-duality hypothesis, hence incorporates that $\dim_\C\Hom_{H(\C)}(E_{\mu_v},\C)=1$ for all $v\in S_\infty$. (See \cite{grob-ragh},  Prop.\ 6.3.1.)

\subsection{Cohomology of Locally symmetric spaces}\label{scet:loc_sym_spaces}
Define the orbifolds
$$\SSS_G:=G(F)\backslash G(\A)/A_G K^\circ_\infty=G(F)\backslash G(\A)^{1}/K^\circ_\infty$$
and
$$\tilde\SSS_H:=H(F)\backslash H(\A)/A_G K^\circ_{H,\infty}.$$
A representation $E_\mu$ as in \S\ref{sect:highweights} defines a locally constant sheaf $\E_\mu$ on $\SSS_G$, whose espace \'etal\'e is $G(\A)^1/K^\circ_\infty\times_{G(F)} E_\mu$ (with the discrete topology on $E_\mu$). Along the closed map $\mathcal J:\tilde\SSS_H\rightarrow\SSS_G$, which is induced by $J$, \S\ref{sect:GH}, we also obtain a sheaf on $\tilde\SSS_H$, which we will again denote by $\E_\mu$. Let $H^q_c(\SSS_G,\E_\mu)$ (resp.\ $H_c^q(\tilde\SSS_H,\E_\mu)$) be the corresponding space of sheaf cohomology with compact support. This is an admissible $G(\A_{f})$-module (resp.\ $H(\A_f)$-module), cf.\ \cite{rohlfs}, Cor.\ 2.13. Observe that the closed map $\mathcal J$ from above gives rise to a non-trivial $H(\A_f)$-equivariant map
$$\mathcal J^q_\mu: H_c^q(\SSS_G,\E_\mu)\ra H_c^q(\tilde\SSS_H,\E_\mu).$$

\subsection{Complex automorphisms and rational structures}\label{sect:sigmatwist}
For $\sigma\in \textrm{Aut}(\C)$, let us define the $\sigma$-twist ${}^\sigma\!\nu$ of an (abstract) representation $\nu$ of $G(\A_f)$ (resp., $G(F_v)$, $v\notin S_\infty$) on a complex vector space $W$, following Waldspurger \cite{waldsp}, I.1: If $W'$ is a $\C$-vector space with a $\sigma$-linear isomorphism $\phi:W\ra W'$ then we set
$${}^\sigma\!\nu:= \phi\circ\nu\circ \phi^{-1}.$$
This definition is independent of $\phi$ and $W'$ up to equivalence of representations, whence we may always take $W':=W\otimes_{\sigma}\C$, i.e., the abelian group $W$ endowed with the scalar multiplication $\lambda\cdot_\sigma w:=\sigma^{-1}(\lambda)w$. \\\\
Furthermore, if $\nu_\infty=\bigotimes_{v\in S_{\infty}}\nu_{v}$ is an (abstract) representation of the real Lie group $G_\infty$, we let $${}^{\sigma}\!\nu_\infty:=\bigotimes_{v\in S_{\infty}}\nu_{\sigma^{-1} v},$$
interpreting $v\in S_{\infty}$ as an embedding of fields $v:F\hookrightarrow\R$. Combining these two definitions, we may define the $\sigma$-twist on a global representation $\nu=\nu_\infty\otimes\nu_f$ of $G(\A)$ as
$${}^\sigma\!\nu:={}^{\sigma}\!\nu_\infty\otimes {}^\sigma\!\nu_f.$$
We recall also the definition of the rationality field of a representation from \cite{waldsp}, I.1. If $\nu$ is any of the representations considered above, then let $\mathfrak S(\nu)$ be the group of all automorphisms $\sigma\in \textrm{Aut}(\C)$ such that ${}^\sigma\!\nu\cong\nu $. Then the rationality field $\Q(\nu)$ of $\nu$ is defined as the fixed-field of $\mathfrak S(\nu)$ within $\C$, i.e.,
$$\Q(\nu):=\{z\in\C| \sigma(z)=z \textrm{ for all } \sigma\in\mathfrak S(\nu)\}.$$
We say that a representation $\nu$ on a $\C$-vector space $W$ is defined over a subfield $\F\subset\C$, if there is a $\F$-vector subspace $W_\F\subset W$, stable under the given action, such that the canonical map $W_\F\otimes_\F\C\rightarrow W$ is an isomorphism. The following lemma is due to Clozel, \cite{clozel}, p.\ 122 and p.\ 128. (See also \cite{grobrag}, Lem.\ 7.1.)

\begin{lem}
Let $E_\mu$ be an irreducible, algebraic representation as in \S \ref{sect:highweights}. As a representation of the diagonally embedded group $G(F)\hra G_\infty$, ${}^\sigma\!E_\mu$ is isomorphic to the abstract representation $E_\mu\otimes_\sigma\C$. Moreover, as a representation of $G(F)$, $E_\mu$ is defined over $\Q(E_\mu)$.
\end{lem}

We fix once and for all a $\Q(E_\mu)$-structure on $E_\mu$ as a representation of $G(F)$. Clearly, this also fixes a $\Q(E_\mu)$-structure on $E_\mu$ as a representation of $H(F)$. As a consequence, the $G(\A_f)$-module $H_c^q(\SSS_G,\E_\mu)$ and the $H(\A_f)$-module $H_c^q(\tilde\SSS_H,\E_\mu)$ carry a fixed, natural $\Q(E_\mu)$-structure, cf.\ \cite{clozel} p.\ 123. Moreover, this also pins down natural $\sigma$-linear, equivariant isomorphisms
\begin{equation}\label{eq:H_mu}
\mathcal H^{\sigma,q}_\mu: H^q_c(\SSS_G,\E_\mu)\ira H^q_c(\SSS_G,{}^\sigma\!\E_\mu)\quad\quad\textrm{and}\quad\quad \tilde{\mathcal H}^{\sigma,q}_\mu: H^q_c(\tilde\SSS_H,\E_\mu)\ira H^q_c(\tilde\SSS_H,{}^\sigma\!\E_\mu)
\end{equation}
for all $\sigma\in$ Aut$(\C)$, cf.\ \cite{clozel} p.\ 128. The following lemma is obvious.

\begin{lem}\label{lem:Qstructures}
For all $\sigma\in {\rm Aut}(\C)$ the following diagram commutes,
$$
\xymatrix{H^q_c(\SSS_G,\E_\mu) \ar[rr]^{\mathcal J^q_\mu} \ar[d]_{\mathcal H^{\sigma,q}_\mu}& & H^q_c(\tilde\SSS_H,\E_\mu)\ar[d]_{\tilde{\mathcal H}^{\sigma,q}_\mu}  \\
H^q_c(\SSS_G,{}^\sigma\!\mathcal E_\mu) \ar[rr]^{\mathcal J^q_{{}^\sigma\!\mu}} && H^q_c(\tilde\SSS_H,{}^\sigma\!\mathcal E_\mu)}
$$
\end{lem}

\section{Facts and conventions for cuspidal automorphic representations}\label{sect:cusprep}
\subsection{Cohomological cusp forms}\label{sect:cusprep1}
In this paper, we let $\Pi$ be an irreducible unitary cuspidal automorphic representation of $G(\A)$ with trivial central character. Furthermore, we assume that $\Pi$ is self-dual, i.e., $\Pi\cong\Pi^\vee$. This is no loss of generality, as the main result will only hold for such cuspidal representations. (Compare this to Prop.\ \ref{prop:lift} below.) Recall that $\Pi$ has a (unique) Whittaker model (with respect to $\psi$). We write
\[
\whit^{\psi}:\Pi\rightarrow\Whit^{\psi}(\Pi)
\]
for the realization of $\Pi$ in its Whittaker model $\Whit^{\psi}(\Pi)$ by the $\psi$-Fourier coefficient. Analogous notation is used locally, which we trust will not cause any confusion. We will furthermore assume that $\Pi$ is cohomological: By this we understand that there is an irreducible, algebraic representation $E_\mu$ of $G_\infty$, as in Sect.\ \ref{sect:highweights}, such that the archimedean component $\Pi_\infty$ of $\Pi$ has non-vanishing $(\m_G,K^\circ_\infty)$-cohomology with respect to $E_\mu$, i.e.,
$$H^q(\m_G,K^\circ_\infty,\Pi_\infty\otimes E_\mu)\neq 0,$$
for some degree $q$.
\begin{lem}\label{lem:cohom}
Let $\rho_\infty$ be an irreducible unitary $(\g_\infty,K_\infty^\circ)$-module with trivial $A_G$-action. Then the following assertions are equivalent:
\begin{enumerate}
\item $H^*(\m_G,K^\circ_\infty,\rho_\infty\otimes E_\mu)\neq 0,$
\item $H^*(\g_\infty,K^\circ_\infty,\rho_\infty\otimes E_\mu)\neq 0,$
\item $H^*(\g_\infty,(Z_\infty K_\infty)^\circ,\rho_\infty\otimes E_\mu)\neq 0.$
\end{enumerate}
\end{lem}
\begin{proof}
This follows combining the following well-known results on relative Lie algebra cohomology :\cite{bowa}, I. 1.3 (the K\"unneth rule), I. 5.1, I. Thm.\ 5.3 (Wigner's lemma) and II. Prop.\ 3.1 (all cochains are closed and harmonic). We refer to \cite{bowa} for details.
\end{proof}
As a consequence, the archimedean component $\Pi_\infty$ of a cuspidal automorphic representation $\Pi$, as above, is cohomological in our sense, if and only if $\Pi_\infty$ has non-vanishing $(\g_\infty,K^\circ_\infty)$-cohomology or equivalently, non-vanishing $(\g_\infty,(Z_\infty K_\infty)^\circ)$-cohomology with respect to the same algebraic, self-dual coefficient module $E_\mu$ (although the degrees and dimensions of non-trivial cohomology spaces may change).

The fundamental group $\pi_0(G_\infty)\cong K_\infty/K^\circ_\infty$ acts on the cohomology groups $H^q(\m_G,K^\circ_\infty,\Pi_\infty\otimes E_\mu)$ in each degree. For a character $\epsilon\in\pi_0(G_\infty)^*$, which we identify with
$$
\epsilon = (\epsilon_1,\dots,\epsilon_d) \in  (\Z/2\Z)^d \cong \pi_0(G_\infty)^*,
$$
one obtains a corresponding $\pi_0(G_\infty)$-isotypic component $H^q(\m_G,K^\circ_\infty,\Pi_\infty\otimes E_\mu)[\epsilon]$.
Put
\begin{equation}\label{eqn:top-degree}
t:=dn(n+1)-1.
\end{equation}
Then,
$$
\dim_\C H^t(\m_G,K^\circ_\infty,\Pi_\infty\otimes E_\mu)[\epsilon]=1
$$
for all $\epsilon\in \pi_0(G_\infty)^*$. This is a direct consequence of the formula in Clozel \cite{clozel}, Lem.\ 3.14 (see also \cite[3.1.2]{mahnk} or \cite[5.5]{grobrag}), the K\"unneth rule (\cite{bowa}, I. 1.3) and the fact that $\s$ is a $d-1$-dimensional abelian real Lie algebra, whence $H^q(\s,\C)\cong\Lambda^q\C^{d-1}$.

Observe furthermore, that (for all degrees $q$ and characters $\epsilon\in \pi_0(G_\infty)^*$) there is a natural $G(\A_f)$-equivariant inclusion
\begin{equation}\label{eq:Delta}
\Delta^q_\Pi: H^{q}(\m_G,K^\circ_\infty,\Pi\otimes E_\mu)[\epsilon]\hookrightarrow H_c^q(\SSS_G,\E_\mu).
\end{equation}
This is well-known and follows from the fact that the cohomology of cuspidal automorphic forms injects into the cohomology $H^q(\SSS_G,\E_\mu)$, while at the same time, $H_c^q(\SSS_G,\E_\mu)$ is isomorphic to the space of cohomology of fast decreasing differential forms (\cite{borel_stablerealII}, \S 5).

\subsection{Rational structures}\label{sect:rat_struct}
We have the following result:

\begin{prop}\label{prop:rational_structures}
Let $\Pi$ be a cuspidal automorphic representation of $G(\A)$, which is cohomological with respect to $E_\mu$. Then, the $\sigma$-twisted representation ${}^\sigma\Pi$ is also cuspidal automorphic and it is cohomological with respect to ${}^\sigma\! E_\mu$. For every $\epsilon\in \pi_0(G_\infty)^*$, the irreducible unitary $G(\A_f)$-module $H^{t}(\m_G,K^\circ_\infty,\Pi\otimes E_\mu)[\epsilon]$ is defined over the rationality field $\Q(\Pi_f)$. This field is a number field, containing $\Q(E_\mu)$.
\end{prop}
\begin{proof}
This is essentially due to Clozel, \cite{clozel}. In order to derive the above result from \cite{clozel}, observe that $\Pi_\infty$ is ``regular algebraic'' in Clozel's sense, if and only if it is cohomological in our sense: This follows using Lem.\ \ref{lem:cohom} and \cite{grobrag} Thm.\ 6.3. Hence, ${}^\sigma\Pi_f$ is the non-archimedean part of a cuspidal automorphic representation, which is cohomological with respect to ${}^\sigma\! E_\mu$ by \cite{clozel} Thm.\ 3.13. By uniqueness, see e.g. \cite{grobrag} 5.5, the archimedean part of this cuspidal automorphic representation is isomorphic to ${}^\sigma\Pi_\infty$ as defined above. By \cite{clozel}, Prop. 3.1, $H^{t}(\m_G,K^\circ_\infty,\Pi\otimes E_\mu)[\epsilon]$ is defined over $\Q(\Pi_f)$ (See also \cite{grobrag}, Cor.\ 8.7.). Finally, it is an implicit consequence of \cite{clozel}, Thm.\ 3.13 and its proof that $\Q(\Pi_f)$ is a number field containing $\Q(E_\mu)$. For a detailed exposition of the latter assertion, we refer to \cite{grobrag}, Thm.\ 8.
1 and the proof of \cite{grobrag}, Cor.\ 8.7.
\end{proof}

\begin{rem}\label{rem:rat_structure}
The $\Q(\Pi_f)$-structure on $H^{t}(\m_G,K^\circ_\infty,\Pi\otimes E_\mu)[\epsilon]$ is unique up to homotheties, i.e., up to multiplication by non-zero complex numbers, cf.\ \cite{clozel} Prop.\ 3.1. As $\Q(E_\mu)\subseteq\Q(\Pi_f)$, we may fix the $\Q(\Pi_f)$-structure on $H^{t}(\m_G,K^\circ_\infty,\Pi\otimes E_\mu)[\epsilon]$ which is induced by $\Delta^t_\Pi$, cf.\ \eqref{eq:Delta}, and our choice of a $\Q(E_\mu)$-structure on $H_c^t(\SSS_G,\E_\mu)$, cf\ \S\ref{sect:sigmatwist}.
\end{rem}

\subsection{Lifts from $\SO(2n+1)$}\label{sect:liftSO}
We resume the assumptions made on $\Pi$ from Sect.\ \ref{sect:cusprep1}. As a last part of notation for $\Pi$, let us introduce $S=S(\Pi,\psi)$, which is a (sufficiently large) finite set of places of $F$, containing $S_\infty$ and such that outside $S$, both $\Pi$ and $\psi$ are unramified.

\begin{prop}\label{prop:lift}
Let $\Pi$ be a cuspidal automorphic representation of $G(\A)$ as in Sect.\ \ref{sect:cusprep1} above. Then the following assertions are equivalent:
\begin{enumerate}
\item The partial exterior square $L$-function, $L^S(s,\Pi,\Lambda^2)$, has a pole at $s=1$,
\item $\Pi$ is the lift of an irreducible unitary generic cuspidal automorphic representation of the split special orthogonal group $\SO(2n+1)$ in the sense of \cite{ckpssh}, \S 1.
\end{enumerate}
\end{prop}
\begin{proof}
With our assumptions on $\Pi$ this is \cite{ckpssh}, Thm.\ 7.1.
\end{proof}
This result is recalled for convenience, as it provides an alternative description of what it means that the exterior square $L$-function of $\Pi$ has a pole at $s=1$. We will have to make this assumption later, in order to obtain our main theorems. See, e.g., Thm.\ \ref{thm:integralrep_extsqr},  Thm.\ \ref{thm:residue} , and Thm.\ \ref{thm:Symm2}. It is not referred to until \S\ref{sect:BG}.
In any case, the above result is accompanied by

\begin{prop}\label{prop:sigma_poles}
Let $\Pi$ be a cuspidal automorphic representation of $G(\A)$ as in \S\ref{sect:cusprep1}. Assume that $\Pi$ satisfies the equivalent conditions of Prop.\ \ref{prop:lift}, i.e., the partial exterior square $L$-function, $L^S(s,\Pi,\Lambda^2)$, has a pole at $s=1$. Then, for all $\sigma\in {\rm Aut}(\C)$, also ${}^\sigma\Pi$ is a cuspidal automorphic representation of $G(\A)$ as in \S\ref{sect:cusprep1}, which satisfies the equivalent conditions of Prop.\ \ref{prop:lift}, i.e., the partial exterior square $L$-function, $L^S(s,{}^\sigma\Pi,\Lambda^2)$, has a pole at $s=1$.
\end{prop}
\begin{proof}
The first assertion, that ${}^\sigma\Pi$ is again a cuspidal automorphic representation of $G(\A)$ as in \S\ref{sect:cusprep1}, has already been proved in Prop.\ \ref{prop:rational_structures}. For the second assertion, observe that the set $S$ does not change under the action of Aut$(\C)$. Now combine \cite{grob-ragh}, Thm.\ 3.6.2 and \cite{jac_shal_AnnArbor}, Thm.\ 1, p.\ 213.
\end{proof}

\section{Top-degree periods}\label{sect:periods}

\subsection{The map $W^\sigma$}\label{sect:Whitt_rational_str}
Recall the Whittaker model $\Whit^{\psi_f}(\Pi_f)=\otimes'_{v\notin S_\infty}\Whit^{\psi_v}(\Pi_v)$ of $\Pi_f=\otimes'_{v\notin S_\infty}\Pi_v$. Given a Whittaker function $\xi_v\in \Whit^{\psi_v}(\Pi_v)$ on $G(F_v)$, $v\notin S_\infty$, and $\sigma\in$ Aut$(\C)$, we may define a Whittaker function ${}^{\sigma}\xi_v\in \Whit^{\psi_v}({}^\sigma\Pi_v)$ by
\begin{equation}
\label{eqn:aut-c-action}
{}^{\sigma}\xi_v(g_v): = \sigma(\xi_v(\t_{\sigma,v}\cdot g_v)),
\end{equation}
where $\t_{\sigma,v}$ is the (uniquely determined) diagonal matrix in $G(\O_v)$, having $1$ as its last entry, which conjugates $\psi_v$ to $\sigma\circ\psi_v$. (Observe that $\t_{\sigma,v}$ does not depend on $\psi_v$).
See \cite{mahnk}, 3.3 and \cite{raghuram-shahidi-imrn}, 3.2. This provides us a $\sigma$-linear intertwining operator
$$W^\sigma:\Whit^{\psi_v}(\Pi_v)\rightarrow\Whit^{\psi_v}({}^\sigma\Pi_v)$$
$$\xi_v\mapsto{}^\sigma\xi_v,$$
for all $\sigma\in$ Aut$(\C)$. In particular, we get a $\Q(\Pi_v)$ structure on $\Whit^{\psi_v}(\Pi_v)$ by taking the subspace of Aut$(\C/\Q(\Pi_v))$-invariant vectors. By the same procedure, we obtain a $\Q(\Pi_f)$ structure on $\Whit^{\psi_f}(\Pi_f)$. (Cf.\ \cite{hardermodsym}, p. 80, \cite{mahnk}, 3.3 or \cite{raghuram-shahidi-imrn}, Lem.\ 3.2.)

\subsection{The map $F^t_\Pi$}\label{sect:generator}
Let $\epsilon_0:=((-1)^{n-1},...,(-1)^{n-1})\in\pi_0(G_\infty)^*$. This choice of a character of the fundamental group is forced upon us by the proof of Thm.\ \ref{hyp} and so we restrict our attention from now on to it. Recall the $\Q(\Pi_f)$-rational structure on $H^{t}(\m_G,K^\circ_\infty,\Pi\otimes E_\mu)[\epsilon_0]$, chosen in Rem.\ \ref{rem:rat_structure} above. This choice of a rational structure also fixes a one-dimensional $\Q(\Pi_f)$-rational structure (as a vector space, or, equivalently, as the trivial $G_\infty$-module) inside the one-dimensional $\C$-vector space $H^{t}(\m_G,K^\circ_\infty,\Pi_\infty\otimes E_\mu)[\epsilon_0]$. In other words, it fixes a generator $[\Pi_\infty]^t$ of $H^{t}(\m_G,K^\circ_\infty,\Pi_\infty\otimes E_\mu)[\epsilon_0]$ up to multiplication by $\Q(\Pi_f)^{\times}$. Any such generator $[\Pi_\infty]^t$ is the pull-back of a generator $[\Whit^{\psi_\infty}(\Pi_\infty)]^t$ of $H^{t}(\m_G,K^\circ_\infty,\Whit^{\psi_\infty}(\Pi_\infty)\otimes E_\mu)[\epsilon_0]$ along $(\whit^{\psi_\infty})^{-1}$. According to \cite{bowa}, II. Prop.\ 3.1, a generator of the latter cohomology space may be written in the form
$$
[\Whit^{\psi_\infty}(\Pi_\infty)]^t=\sum_{\underline i=(i_1,...,i_{t})}\sum_{\alpha=1}^{\dim E_\mu} X^*_{\underline i}\otimes\xi^{\epsilon_0}_{\infty,\underline i, \alpha}\otimes e_\alpha,
$$
where
\begin{enumerate}
\item $X^*_{\underline i}=X^*_{i_1}\wedge...\wedge X^*_{i_{t}}\in \Lambda^{t} \left(\m_G/\k_\infty\right)^*$, with respect to a given basis $\{X_j\}$ of $\m_G/\k_\infty$, which fixes the dual-basis $\{X_j^*\}$ of $\left(\m_G/\k_\infty\right)^*$.
\item $e_\alpha=\otimes_{v\in S_\infty} e_{\alpha,v}\in E_\mu=\otimes_{v\in S_\infty}E_{\mu_v}$, such that $\{e_{\alpha,v}\}$ defines a basis of $E_{\mu_v}$ for all $v\in S_\infty$.
\item $\xi^{\epsilon_0}_{\infty,\underline i, \alpha}=\otimes_{v\in S_\infty} \xi^{\epsilon_0}_{v,\underline i, \alpha}\in \Whit^{\psi_\infty}(\Pi_\infty)=\otimes_{v\in S_\infty} \Whit^{\psi_v}(\Pi_v)$, for each $\underline i$ and $\alpha$.
\end{enumerate}
(We may and will assume that $\{X_j\}$ is the extension of a given ordered basis $\{Y_j\}$ of $\m_H/\k_{H,\infty}$ along our embedding $J:H\hra G$. This assumption, however, will only be important later on. See, e.g., \S\ref{sect:deRham} and \S\ref{sect:archimedean}.)
As a consequence of the above discussion, we may write the generator $[\Pi_\infty]^t$ of $H^{t}(\m_G,K^\circ_\infty,\Pi_\infty\otimes E_\mu)[\epsilon_0]$, which is defined (up to $\Q(\Pi_f)^{\times}$) by our choice of a $\Q(\Pi_f)$-rational structure on $H^{t}(\m_G,K^\circ_\infty,\Pi\otimes E_\mu)[\epsilon_0]$, in the form

$$
[\Pi_\infty]^t=\sum_{\underline i=(i_1,...,i_{t})}\sum_{\alpha=1}^{\dim E_\mu} X^*_{\underline i}\otimes (\whit^{\psi_\infty})^{-1}(\xi^{\epsilon_0}_{\infty,\underline i, \alpha})\otimes e_\alpha.
$$

Recall that $\sigma\in {\rm Aut}(\C)$ acts on objects at infinity, which are parameterized by $S_{\infty}$, by permuting the archimedean places. Having given a generator $[\Whit^{\psi_\infty}(\Pi_\infty)]^t$ of the one-dimensional space $H^{t}(\m_G,K^\circ_\infty,\Whit^{\psi_\infty}(\Pi_\infty)\otimes E_\mu)[\epsilon_0]$ hence provides us with a natural choice of a generator $[\Whit^{\psi_\infty}({}^\sigma\Pi_\infty)]^t$ of $H^{t}(\m_G,K^\circ_\infty,\Whit^{\psi_\infty}({}^\sigma\Pi_\infty)\otimes {}^\sigma\!E_\mu)[\epsilon_0]$:
$$[\Whit^{\psi_\infty}({}^\sigma\Pi_\infty)]^t=\sum_{\underline i=(i_1,...,i_{t})}\sum_{\alpha=1}^{\dim E_\mu} X^*_{\underline i}\otimes{}^\sigma\xi^{\epsilon_0}_{\infty,\underline i, \alpha}\otimes {}^\sigma\! e_\alpha,$$
where ${}^\sigma\xi^{\epsilon_0}_{\infty,\underline i, \alpha}=\otimes_{v\in S_\infty} \xi^{\epsilon_0}_{\sigma^{-1} v,\underline i, \alpha}$ (observe that $\epsilon_0$ does not change, when its local components are permuted) and ${}^\sigma\! e_\alpha=\otimes_{v\in S_\infty} e_{\alpha,\sigma^{-1} v}$. In particular, we obtain a generator
$$[{}^\sigma\Pi_\infty]^t=\sum_{\underline i=(i_1,...,i_{t})}\sum_{\alpha=1}^{\dim E_\mu} X^*_{\underline i}\otimes (\whit^{\psi_\infty})^{-1}({}^\sigma\xi^{\epsilon_0}_{\infty,\underline i, \alpha})\otimes {}^\sigma\! e_\alpha$$
of $H^{t}(\m_G,K^\circ_\infty,{}^\sigma\Pi_\infty\otimes {}^\sigma\!E_\mu)[\epsilon_0]$. Observe that this entails the description of an isomorphism of $G(\A_f)$-modules
$$F^t_\Pi:\Whit^{\psi_f}(\Pi_f)\ira H^{t}(\m_G,K_{\infty}^\circ, \Pi\otimes \!E_{\mu})[\epsilon_0]$$.
$$\xi_f\mapsto F^t_\Pi(\xi_f):=\sum_{\underline i=(i_1,...,i_{t})}\sum_{\alpha=1}^{\dim E_\mu} X^*_{\underline i}\otimes \varphi_{\underline i, \alpha}\otimes e_\alpha,$$
where $\varphi_{\underline i, \alpha}=(\whit^\psi)^{-1}(\xi^{\epsilon_0}_{\infty,\underline i, \alpha}\otimes\xi_f)$. Of course, in light of Prop.\ \ref{prop:rational_structures}, the same assertion holds for ${}^\sigma\Pi$. (As a selection of further reference for this section, see also \cite{mahnk}, 3.3 \& 5.1.4, \cite{raghuram-imrn} 3.2.5, and \cite{grob-ragh}, 4.1.)

\subsection{The map $H^{\sigma,t}_\mu$}\label{sect:H^sigma}
As a last ingredient in this section, we define a $\sigma$-linear, $G(\A_f)$-equivariant isomorphism
$$H^{\sigma,t}_\mu:H^{t}(\m_G,K_{\infty}^\circ, \Pi\otimes E_{\mu})[\epsilon_0]\ira H^{t}(\m_G,K_{\infty}^\circ, {}^\sigma\Pi\otimes{}^\sigma\! E_\mu)[\epsilon_0].$$
To that end, recall the embedding $\Delta^q_{\Pi}$ from \eqref{eq:Delta} and the $\sigma$-linear isomorphism $\mathcal H^{\sigma,q}_\mu$ from \eqref{eq:H_mu}. Observe that Im$(\mathcal H^{\sigma,t}_\mu\circ\Delta^t_\Pi)=$Im$(\Delta^t_{{}^\sigma\Pi})$. Indeed,
by Multiplicity One and Strong Multiplicity One for the discrete automorphic spectrum of $G(\A)$, the ${}^\sigma\Pi_f$-isotypic component
of the $G(\A_f)$-module $H^t_c(\SSS_G,{}^\sigma\!\E_\mu)$ is precisely the image of $H^t(\m_G,K^\circ_\infty,{}^\sigma\Pi\otimes {}^\sigma\!E_\mu)$. As the natural
action of $\pi_0(G_\infty)$ and of $G(\A_f)$ on $H^t_c(\SSS_G,\E_\mu)$ commute, this shows that Im$(\mathcal H^{\sigma,t}_\mu\circ\Delta^t_\Pi)=$Im$(\Delta^t_{{}^\sigma\Pi})$ as claimed. Since $\Delta^t_{{}^\sigma\Pi}$ is injective, the map 
$$H^{\sigma,t}_\mu:=(\Delta^t_{{}^\sigma\Pi})^{-1}\circ\mathcal H^{\sigma,t}_\mu\circ\Delta^t_\Pi$$
is hence a well-defined $\sigma$-linear, $G(\A_f)$-equivariant isomorphism mapping $H^{t}(\m_G,K_{\infty}^\circ, \Pi\otimes E_{\mu})[\epsilon_0]$ onto $H^{t}(\m_G,K_{\infty}^\circ, {}^\sigma\Pi\otimes{}^\sigma\! E_\mu)[\epsilon_0]$ as desired. (Shortly speaking, this amounts to say that the
restriction $H^{\sigma,t}_\mu$ of $\mathcal H^{\sigma,t}_\mu$ to the submodule $H^{t}(\m_G,K_{\infty}^\circ, \Pi\otimes E_{\mu})[\epsilon_0]$ of
$H^t_c(\SSS_G,\E_\mu)$ has image $H^{t}(\m_G,K_{\infty}^\circ, {}^\sigma\Pi\otimes{}^\sigma\! E_\mu)[\epsilon_0]$.)

\subsection{Top-degree Whittaker Periods}\label{sect:periods_sub}
Recall the maps $W^\sigma$ (\S\ref{sect:Whitt_rational_str}), $ F^t_\Pi$ (\S\ref{sect:generator}) and $H^{\sigma,t}_\mu$ (\S\ref{sect:H^sigma}). There is the following result:
\begin{prop}\label{prop:periods}
For every $\sigma\in {\emph{Aut}}(\C)$, there is a non-zero complex number $p^t({}^\sigma\Pi)$ (a ``period''), uniquely determined up to multiplication by elements in $\Q({}^\sigma\Pi_f)^\times$, such that the normalized maps $\mathcal F^t_{{}^\sigma\Pi}:=p^t({}^\sigma\Pi)^{-1} F^t_{{}^\sigma\Pi}$ make the following diagram commutative:
$$
\xymatrix{
\Whit^{\psi_f}(\Pi_f) \ar[rrrr]^{\mathcal F^t_\Pi}\ar[d]_{W^\sigma} & & & &
H^{t}(\m_G,K_{\infty}^\circ, \Pi\otimes E_{\mu})[\epsilon_0]
\ar[d]^{H^{\sigma,t}_\mu} \\
\Whit^{\psi_f}({}^\sigma\Pi_f)
\ar[rrrr]^{\mathcal F^t_{{}^\sigma\Pi}}
& & & &
H^{t}(\m_G,K_{\infty}^\circ, {}^\sigma\Pi\otimes {}^\sigma\! E_\mu)[\epsilon_0]
}
$$
\end{prop}
\begin{proof}
This us essentially due to the uniqueness of essential vectors for $\Pi_v$, $v\notin S_\infty$: Otherwise put, the proof of Prop./Def. 3.3 in Raghuram--Shahidi \cite{raghuram-shahidi-imrn} goes through word for word in our (slightly different) situation at hand.
\end{proof}

The (Whittaker) periods $p^t(\Pi)$ defined by Prop.\ \ref{prop:periods} are the analogues of the (Shalika) periods $\omega^\epsilon(\Pi_f)$ defined in Grobner--Raghuram Def.\/Prop.\ 4.2.1. The idea behind the construction of $p^t(\Pi)$ (as of $\omega^\epsilon(\Pi_f)$), however, goes back to \cite{hardermodsym}, \cite{mahnk} and \cite{raghuram-shahidi-imrn}.

\section{An Aut$(\C)$-rational assignment for Whittaker functions}\label{sect:diagram}
\subsection{The map $T_\mu$}\label{sect:T}
We now want to push the considerations from the previous sections even further. Let $E_\mu=\otimes_{v\in S_\infty}E_{\mu_v}$ be an irreducible, algebraic representation as in \S\ref{sect:highweights}. We have $\dim_\C\Hom_{H(\C)}(E_{\mu_v},\C)=1$ for all $v\in S_\infty$. Let us fix $T_{\mu_v}\in \Hom_{H(\C)}(E_{\mu_v},\C)$ and set $T_\mu:=\otimes_{v\in S_\infty}T_{\mu_v}\in \Hom_{H(\C)}(E_{\mu},\C)$. For $\sigma\in$Aut$(\C)$, we obtain $T_{{}^\sigma\! \mu}=\otimes_{v\in S_\infty}T_{\mu_{\sigma^{-1}v}}\in \Hom_{H(\C)}({}^\sigma\!E_{\mu},\C)$. The map induced on cohomology,
$$H^t_c(\tilde\SSS_H,{}^\sigma\!\E_\mu)\ra H^t_c(\tilde\SSS_H,\C)$$
will be denoted by the same letter $T_{{}^\sigma\! \mu}$.

\subsection{The de-Rham-isomorphism $\mathcal R$}\label{sect:deRham}
So far, we have not made any choice of a Haar measure on $H(\A_f)$. From this section on, we will restrict our possible choices on $\Q$-valued Haar measures on $H(\A_f)$. In \S\ref{sect:measures} we will specify our concrete choice of a measure in details. (So far, this is not necessary.) Let $dh_f$ be any $\Q$-valued Haar measure of $H(\A_f)$. It is important to notice that we have $\dim_\R\tilde\SSS_H=dn(n+1)-1=t$, cf.\ \S\ref{eqn:top-degree}, because we assumed that $F$ is totally real. Knowing this, a short moment of thought shows that we obtain a surjective map $\mathcal R: H^t_c(\tilde\SSS_H,\C)\ra\C$, induced by the de Rham-isomorphism: Indeed, let $K_{f}$ be a compact open subgroup of $H(\A_f)$ and set
$$\tilde\SSS^{K_{f}}_H:=H(F)\backslash H(\A)/A_G K^\circ_{H,\infty} K_{f}.$$
Then it is easy to see that $\tilde\SSS_H$ is homeomorphic to the projective limit
$$\tilde\SSS_H\cong\varprojlim_{K_{f}} \tilde\SSS^{K_{f}}_H$$
running over the compact open subgroups $K_{f}$ of $H(\A_f)$, partially ordered by opposite inclusion, \cite{rohlfs} Prop.\ 1.9. As $\dim_\R\tilde\SSS^{K_{f}}_H=t$ for all $K_f$, we may use the de-Rham-isomorphism to define a surjective map $H^t_c(\tilde\SSS^{K_f}_H,\C)\ra\C$. More precisely, each of the (finitely many, cf.\ \cite{bor1}, Thm.\ 5.1) connected components of $\tilde\SSS^{K_f}_H$ is homeomorphic to a quotient of $H^\circ_\infty/A_G K^\circ_{H,\infty}$ by a discrete subgroup of $H(F)$.
Recall the ordered basis $\{Y_j\}$ of $\m_H/\k_{H,\infty}$, from \S\ref{sect:generator}. It determines a choice of an orientation on $H^\circ_\infty/A_G K^\circ_{H,\infty}$, whence on each connected component of $\tilde\SSS^{K_f}_H$ and so finally also on $\tilde\SSS^{K_f}_H$. Hence, the de-Rham-isomorphism provides us a surjection
$$R^{K_f}: H^t_c(\tilde\SSS^{K_f}_H,\C)\ra\C.$$
The normalized maps, $\mathcal R^{K_f}:={\rm vol}_{dh_f}(K_f)\cdot R^{K_f}$ form a system of compatible maps with respect to the pull-backs, given by the coverings $\tilde\SSS^{K_{f}}_H\twoheadrightarrow\tilde\SSS^{K'_{f}}_H$, $K_f\subseteq K'_f$. Hence, taking the inductive limit, we obtain a surjection $\varinjlim_{K_f}\mathcal R^{K_f}:\varinjlim_{K_f} H^t_c(\tilde\SSS^{K_f}_H,\C)\ra\C.$ Since $H^t_c(\tilde\SSS_H,\C)\cong\varinjlim_{K_f} H^t_c(\tilde\SSS^{K_f}_H,\C)$, see \cite{rohlfs} Cor.\ 2.12, this finally defines a surjection
$$\mathcal R:H^t_c(\tilde\SSS_H,\C)\ra\C,$$
as mentioned above. (Compare this also to the considerations in  \cite{grob-ragh}, 6.4, \cite{grob_harris}, 3.8, \cite{grob_harris_lapid}, 5.1 and \cite{raghuram-imrn}, 3.2.3 as well as to the corresponding references therein.)

\subsection{In summary: A rational diagram}\label{sect:diagram_sub}
In the following proposition, we abbreviate $H^t(\m_G,K^\circ_\infty,\Pi\otimes E_\mu)[\epsilon_0]$ by $H^t(\Pi\otimes E_\mu)[\epsilon_0]$ (with analogous notation for the cohomology of the $\sigma$-twisted representations). Recollecting what we observed in \S\ref{sect:Whitt_rational_str} -- \S\ref{sect:deRham}, and we find

\begin{prop}\label{prop:diagram}
The following diagram is commutative:
\begin{equation}\label{diagram}
\small
\xymatrix{
\Whit^{\psi_f}(\Pi_f) \ar[r]^{\mathcal F^t_\Pi}\ar[d]_{W^\sigma} & H^{t}(\Pi\otimes E_{\mu})[\epsilon_0]\ar@{^{(}->}[r]^{\quad\Delta^t_\Pi}
\ar[d]^{H^{\sigma,t}_\mu} & H^t_c(\SSS_G,\E_\mu) \ar[r]^{\mathcal J^t_\mu}\ar[d]^{\mathcal H^{\sigma,t}_\mu} & H^t_c(\tilde\SSS_H,\E_\mu) \ar[r]^{T_\mu}\ar[d]^{\tilde{\mathcal{H}}^{\sigma,t}_\mu} & H^t_c(\tilde\SSS_H,\C) \ar[r]^{\quad\quad\mathcal R}\ar[d]^{\tilde{\mathcal{H}}^{\sigma,t}_0} & \C \ar[d]^{\sigma}\\
\Whit^{\psi_f}({}^\sigma\Pi_f)\ar[r]^{\mathcal F^t_{{}^\sigma\Pi}} & H^{t}({}^\sigma\Pi\otimes {}^\sigma\! E_\mu)[\epsilon_0]\ar@{^{(}->}[r]^{\quad\Delta^t_{{}^\sigma\Pi}} & H^t_c(\SSS_G,{}^\sigma\!\E_\mu)\ar[r]^{\mathcal J^t_{{}^\sigma\!\mu}} & H^t_c(\tilde\SSS_H,{}^\sigma\!\E_\mu) \ar[r]^{T_{{}^\sigma\!\mu}} & H^t_c(\tilde\SSS_H,\C)\ar[r]^{\quad\quad\mathcal R} & \C
}
\normalsize
\end{equation}
Its horizontal arrows are linear, whereas its vertical arrows are $\sigma$-linear. 
\end{prop}
\begin{proof}
The first square from the left is commutative by Prop.\ \ref{prop:periods}, while the second square is commutative by the definition of $H^{\sigma,t}_\mu$ in \S\ref{sect:H^sigma}. Commutativity of the third square is the assertion of Lem.\ \ref{lem:Qstructures}. The fourth square commutes by the very definition of $T_{{}^\sigma\!\mu}$ in \S\ref{sect:T}, while commutativity of the last square is due to the $\Q$-rationality of the measure on $H(\A_f)$, \S\ref{sect:deRham}.
\end{proof}

\section{An integral representation of the residue of the exterior square $L$-function}\label{sect:intrep}
In this section, we will recapitulate some results from Jacquet--Shalika \cite{jacshal} and Bump--Friedberg \cite{bumbginzburg}.

\subsection{Eisenstein series and a result of Jacquet--Shalika}\label{sect:SchwartzBruhat}
We resume the notation and assumptions made in the previous sections. In addition, for any integer $m\geq 2$, we will now fix once and for all a Schwartz--Bruhat function $\Phi=\otimes_v\Phi_v\in\mathscr S(\A^m)$: We assume that $\Phi_v$ is the characteristic function of $\O^m_v$ at all $v\notin S_\infty$, while at the archimedean places $v\in S_\infty$, we assume to have chosen ($O(m)$-finite) local components $\Phi_v$, such the global Schwartz--Bruhat function $\Phi$ satisfies $\hat\Phi(0)\neq 0$. Here, we wrote
$$\hat\Phi(x):=\int_{\A^n}\Phi(y) \psi ({}^t y \cdot x)dy$$
for the Fourier transform of $\Phi$ (at $x$) with respect to the self-dual Haar measure $dy$ on $\A^m$, i.e., the unique Haar measure on $\A^m$ which satisfies $\hat{\hat\Phi}(x)=\Phi(-x)$ for all $x\in\A^m$. Let
$$f_{v,s}(g_v):=|\det(g_v)|^s_v\int_{F^\times_v}\Phi_v(t\cdot (0,...,0,1)g_v)|t|^{ms}_vd^\times t$$
and
$$f_s(g):=\otimes_v f_{v,s}(g_v)=\|\det(g)\|^s\int_{\A^\times}\Phi(t\cdot (0,...,0,1)g)\|t\|^{ms}d^\times t$$
for $\Re e(s)\gg 0 $. Then $f_s\in{\rm Ind}^{\GL_m(\A)}_{\GL_{m-1}(\A)\times\GL_1(\A)}[\delta^s_P]$, (unnormalised parabolic induction) where
$$\delta_P\left(\!\!\left(\!\! \begin{array}{ccc}
h &  0\\
0 &  a
\end{array}\!\!\right)\!\!\right)=\|\det(h)\|\cdot\|a\|^{-(m-1)}$$
is the modulus character of the standard parabolic subgroup $P$ of $\GL_m$, with Levi subgroup $L=\GL_{m-1}\times\GL_1$. Clearly, the analogous assertion holds for the local components $f_{v,s}$. %Moreover, with our particular choice for $\Phi$, $f_{v,1}(id_v)=1$, for all $v\notin S_\infty$.
There is the following result due to Jacquet--Shalika, \cite{jacshal}, Lem.\ 4.2

\begin{lem}\label{lem:Eisensteinseries}
The Eisenstein series associated with $f_s$, formally defined as
$$E(f_s,\Phi)(g):=\sum_{\gamma\in P(F)\backslash \GL_m(F)}f_s(\gamma g),$$
extends to a meromorphic function on $\Re e(s) > 0$. It has a simple pole at $s=1$ with constant residue
$${\rm Res}_{s=1}(E(f_s,\Phi)(g))= c_m\cdot\hat\Phi(0).$$
Here, $c_m$ is a certain non-zero complex number.
\end{lem}

\subsection{Measures}\label{sect:measures}
When dealing with rationality results of special values of $L$-functions, the choice of measures is all-important. In this section, we specify our choices of measures, which will be guided by the explicit choices made in Bump--Friedberg \cite{bumbginzburg}.

Let $m\geq 2$ be again any integer and consider the group $\GL_m/F$. A measure $dg$ of $\GL_m(\A)$ will be a product $dg=\prod_v dg_v$ of local Haar measures of $\GL_m(F_v)$. We write $dg^{BF}_v$ for the local Haar measure of $\GL_m(F_v)$ chosen in Bump--Friedberg, \cite{bumbginzburg}. See {\it loc.\ cit.\ }(3.2), p.\ 61. At $v\notin S_\infty$, these measures assign rational volumes to compact open subgroups of $\GL_m(F_v)$. Furthermore, the product measure $dg^{BF}_f:=\prod_{v\notin S_\infty}dg^{BF}_v$ gives rational volumes to compact open subgroups of $\GL_m(\A_f)$.

At $v\notin S_\infty$, we define our choice of a measure to be the one of Bump--Friedberg,
$$dg_v:=dg^{BF}_v$$
whereas at an archimedean place $v\in S_\infty$, we let $dg_v$ be the local Haar measures of $\GL_m(\R)$ such that $SO(m)$ has volume $1$. If we let $m=2n$ as in \S\ref{sect:GH}, we hence have chosen a measure $dg$ on $G(\A)=\GL_{2n}(\A)$. Observe that the volume ${\rm vol}_{dg}(Z(F)\backslash Z(\A)/A_G)$ is well-defined and finite.

Recall the group $H=\GL_n\times\GL_n$, \S\ref{sect:GH}. We will use the notation $(g,g')$, to specify an element of $H(\A)$ (and use analogous notation locally). A measure of $H(\A)$ will be the product of a measures $dg$ and $dg'$ as chosen above for $m=n$ of the two isomorphic copies of $\GL_n(\A)$ inside $H(\A)$. As $Z\subset H$, also the volume ${\rm vol}_{dg\times dg'}(Z(F)\backslash Z(\A)/A_G)$ is well-defined and finite.

\subsection{A result of Bump--Friedberg}\label{sect:BG}
Let $U_n$ be the group of upper triangular matrices in $\GL_n$, having $1$ on the diagonal and let $Z_n$ be the centre of $\GL_n$. Recall the finite set of places $S=S(\Pi,\psi)$ from \S\ref{sect:liftSO}. By assumption, outside $S$, both $\Pi$ and $\psi$ are unramified (and $\psi$ normalised). Let $\xi=\otimes_v\xi_v\in\Whit^\psi(\Pi)\cong\otimes_v\Whit^{\psi_v}(\Pi_v)$ be a Whittaker function, such that for $v\notin S$, $\xi_v$ is invariant under $G(\O_v)$ and normalized such that $\xi_v(id_v)=1$. Recall the section $f_s=\otimes_v f_{v,s}$ from \S\ref{sect:SchwartzBruhat}, defined by the choice of a Schwartz-Bruhat function $\Phi=\otimes_v\Phi_v$, where we now let $m=n$. Following Bump-Friedberg \cite{bumbginzburg}, p.\ 53, we define the integral
$$Z(\xi,f_{s}):=\int_{U_n(\A)\backslash\GL_n(\A)}\int_{Z_n(\A)U_n(\A)\backslash\GL_n(\A)}\xi(J(g,g'))\ f_{s}(g) \ dg \ dg'.$$
It factors over all places of $F$ as $Z(\xi,f_s)=\prod_v Z_v(\xi_v,f_{v,s})$, where
$$Z_v(\xi_v,f_{v,s}):=\int_{U_n(F_v)\backslash\GL_n(F_v)}\int_{Z_n(F_v)U_n(F_v)\backslash\GL_n(F_v)}\xi_v(J(g_v,g'_v)) \ f_{v,s}(g_v) \ dg_v \ dg'_v.$$
Recall the value $L^S(n,\triv)$ of the partial $L$-function of the trivial character $\triv$ of $\A^\times$ at $n$. Since we assumed that $n\geq 2$, this number is well-defined and non-zero. The following result is crucial for us:

\begin{thm}\label{thm:integralrep_extsqr}
Let $\Pi$ be a cuspidal automorphic representation of $G(\A)$ as in Sect.\ \ref{sect:cusprep1}. Let $\varphi:=(\whit^{\psi})^{-1}(\xi)\in\Pi$ be the inverse image of a Whittaker function $\xi$ as in \S\ref{sect:BG} above and assume that $\Pi$ satisfies the equivalent conditions of Prop.\ \ref{prop:lift}, i.e., the partial exterior square $L$-function, $L^S(s,\Pi,\Lambda^2)$, has a pole at $s=1$. Then,
$$c_n\cdot\hat\Phi(0) \int_{Z(\A)H(F)\backslash H(\A)}\varphi(J(g,g')) \ dg \ dg' = \frac{L^S(\tfrac12,\Pi)\cdot {\rm Res}_{s=1}(L^S(s,\Pi,\Lambda^2))}{L^S(n,\triv)^2}\cdot \prod_{v\in S}\frac{Z_v(\xi_v,f_{v,1})}{L(n,\triv_v)}.$$
(Here, $L^S(\tfrac12,\Pi)$ is the partial principal $L$-function of $\Pi$ at $s'=\tfrac12$.)
\end{thm}
\begin{proof}
With our preparatory work, this is a direct consequence of our choice of measures in \S\ref{sect:measures} and the three main results of Bump--Ginzburg \cite{bumbginzburg}, Thm.\ 1, Thm.\ 2 and Thm.\ 3. Indeed, our Lem.\ \ref{lem:Eisensteinseries} together with \cite{bumbginzburg}, Thm.\ 1 and Thm.\ 2, identify the left hand side with
$$c_n\cdot\hat\Phi(0) \int_{Z(\A)H(F)\backslash H(\A)}\varphi(J(g,g')) \ dg \ dg' = {\rm Res}_{s=1}\left(\frac{Z(\xi,f_{s})}{L(n,\triv)}\right),$$
where $L(n,\triv)=\prod_v L(n,\triv_v)$ is the global $L$-function of the trivial character $\triv$ of $\A^\times$ at $n$. As by assumption $n\geq 2$, $L(n,\triv)$ is well-defined and non-zero. Factorizing $Z(\xi,f_{s})$ as in \S\ref{sect:BG}, and using the description of $Z_v(\xi_v,f_{v,s})$, $v\notin S$, in \cite{bumbginzburg}, Thm.\ 3,
$$Z_v(\xi_v,f_{v,s})=\frac{L(\tfrac12,\Pi_v)\cdot L(s,\Pi_v,\Lambda^2)}{L(n,\triv_v)},$$
we obtain
$${\rm Res}_{s=1}\left(\frac{Z(\xi,f_{s})}{L(n,\triv)}\right)=\frac{L^S(\tfrac12,\Pi)\cdot {\rm Res}_{s=1}(L^S(s,\Pi,\Lambda^2))}{L^S(n,\triv)^2}\cdot \prod_{v\in S}\frac{Z_v(\xi_v,f_{v,1})}{L(n,\triv_v)},$$
since by assumption $L^S(s,\Pi,\Lambda^2)$ carries the (simple) pole of the above expression.
\end{proof}

\subsection{Consequences for the $\sigma$-twisted case}\label{sect:conseq_sigma}
Let $\Pi$ be a cuspidal automorphic representation of $G(\A)$ as in \S\ref{sect:cusprep1} and assume that the partial exterior square $L$-function, $L^S(s,\Pi,\Lambda^2)$, has a pole at $s=1$. Then by Prop.\ \ref{prop:sigma_poles}, ${}^\sigma\Pi$ satisfies the same conditions. Hence, we see that once $\Pi$ satisfies the assumptions made in Thm.\ \ref{thm:integralrep_extsqr}, then automatically also ${}^\sigma\Pi$ satisfies them, i.e., Thm.\ \ref{thm:integralrep_extsqr} holds for the whole Aut$(\C)$-orbit of $\Pi$.

As we are going to us this in the proof of the main results, let us render this more precise. Let $\xi=\otimes_v\xi_v\in\Whit^\psi(\Pi)\cong\otimes_v\Whit^{\psi_v}(\Pi_v)$ be a Whittaker function, such that for $v\notin S$, $\xi_v$ is invariant under $G(\O_v)$ and normalized such that $\xi_v(id_v)=1$. Given $\sigma\in$Aut$(\C)$, let ${}^\sigma\xi\in\Whit^\psi({}^\sigma\Pi)$ be the $\sigma$-twisted Whittaker function, cf.\ \S\ref{sect:Whitt_rational_str} (the action of $\sigma$ on the archimedean part of $\xi$ being by permutations as in \S\ref{sect:generator}), and let  ${}^\sigma\!\varphi:=(\whit^{\psi})^{-1}({}^\sigma\xi)\in{}^\sigma\Pi$ be the corresponding cuspidal automorphic form. Recall our Schwartz-Bruhat function $\Phi\in\mathscr S(\A^n)$ from \S\ref{sect:SchwartzBruhat}, with $m=n$ now. We define the constant
$$c_n(\Phi,\sigma):=\sigma\left(\frac{\hat\Phi(0) \ {\rm vol}_{dg\times dg'}(Z(F)\backslash Z(\A)/A_G)}{c_n}\right)\cdot \frac{c_n}{\hat\Phi(0) \ {\rm vol}_{dg\times dg'}(Z(F)\backslash Z(\A)/A_G)}.$$
This is done purely for cosmetic reasons, as it will become clear below (see the proof of Thm.\ \ref{thm:residue}). By \S\ref{sect:SchwartzBruhat}, $c_n(\Phi,\sigma)$ is non-zero. Let ${}^\sigma\Phi\in\mathscr S(\A^n)$ be the Schwartz-Bruhat function which is defined as follows: At $v\notin S_\infty$, $({}^\sigma\Phi)_v:=\Phi_v$, whereas $({}^\sigma\Phi)_\infty:= c_n(\Phi,\sigma)^{-1}\cdot\Phi_\infty$. As in \S\ref{sect:SchwartzBruhat}, we obtain a function ${}^\sigma\! f_s=\otimes_v {}^\sigma\! f_{v,s}\in{\rm Ind}^{\GL_n(\A)}_{\GL_{n-1}(\A)\times\GL_1(\A)}[\delta^s_P]$ and an associated Eisenstein series $E({}^\sigma\! f_s,{}^\sigma\Phi)$. Clearly, $E({}^\sigma\! f_s,{}^\sigma\Phi)$ satisfies the assertions of Lem.\ \ref{lem:Eisensteinseries}, with $\Phi$ being replaced by ${}^\sigma\Phi$.

In summary, with this notation, saying that Thm.\ \ref{thm:integralrep_extsqr} holds for the whole Aut$(\C)$-orbit of $\Pi$, amounts to the equation
\begin{equation}\label{eq:twist}
c_n{}^\sigma\hat\Phi(0) \int_{Z(\A)H(F)\backslash H(\A)}{}^\sigma\!\varphi(J(g,g')) \ dg \ dg' = \frac{L^S(\tfrac12,{}^\sigma\Pi)\cdot {\rm Res}_{s=1}(L^S(s,{}^\sigma\Pi,\Lambda^2))}{L^S(n,\triv)^2} \prod_{v\in S}\frac{Z_v({}^\sigma\xi_v,{}^\sigma\!f_{v,s})}{L(n,\triv_v)}.
\end{equation}

\section{A rationality result for the exterior square $L$-function}

\subsection{Archimedean considerations}\label{sect:archimedean}
The integral representation of the exterior square $L$-function in Thm.\ \ref{thm:integralrep_extsqr} allows us to combine the results of \S\ref{sect:diagram} and \S\ref{sect:intrep}. Before we derive out first main result, we need a non-vanishing theorem, which is an application of Sun's main result in \cite{sun}.

Recall the generator $[\Whit^{\psi_\infty}(\Pi_\infty)]^t=\sum_{\underline i=(i_1,...,i_{t})}\sum_{\alpha=1}^{\dim E_\mu} X^*_{\underline i}\otimes\xi^{\epsilon_0}_{\infty,\underline i, \alpha}\otimes e_\alpha$ of the cohomology space $H^{t}(\m_G,K^\circ_\infty,\Whit^{\psi_\infty}(\Pi_\infty)\otimes E_\mu)[\epsilon_0]$ from \S\ref{sect:generator}. Recall furthermore, that the basis $\{X^*_j\}$ of $\m_G/\k_\infty$ was the extension of a given ordered basis $\{Y^*_j\}^t_{j=1}$ of $\m_H/\k_{H,\infty}$, whence, for each multi-index $\underline i$ there is a well-defined complex number $s(\underline i)$, such that the restriction of $X^*_{\underline i}$ to $\Lambda^t(\m_H/\k_{H,\infty})^*$ along the injection $\Lambda^t(\m_H/\k_{H,\infty})^*\hra\Lambda^t(\m_G/\k_\infty)^*$, induced by $J$, equals $s(\underline i)\cdot Y_1\wedge..\wedge Y_t$. As a last ingredient, before we can state the aforementioned non-vanishing theorem, we need the following lemma:

\begin{lem}\label{lem:zinfty}
For all $v\in S_\infty$, and $K^\circ_v$-finite $\xi_v\in\Whit^{\psi_v}(\Pi_v)$, the integrals
$$Z_v(s',\xi_{v},f_{v,1}):=\int_{U_n(F_v)\backslash\GL_n(F_v)}\int_{Z_n(F_v)U_n(F_v)\backslash\GL_n(F_v)}\xi_v(J(g_v,g'_v))f_{v,1}(g_v)\Bigg|\frac{\det(g'_v)}{\det(g_v)}\Bigg|_v^{s'-1/2} dg_v dg'_v$$
are a holomorphic multiple (in $s'$) of the local archimedean $L$-function $L(s',\Pi_v)$.
\end{lem}
\begin{proof}
This follows combining Thm.\ \ref{thm:integralrep_extsqr} with \cite{friedjac}, Prop.\ 2.3 and Prop.\ 3.1 {\it loc.\ cit.\ }.
\end{proof}
It follows that the factor $Z_\infty(\xi^{\epsilon_0}_{\infty,\underline i, \alpha},f_{\infty,1}):=\prod_{v\in S_\infty}Z_v(\xi^{\epsilon_0}_{v,\underline i, \alpha},f_{v,1})$ of the product $\prod_{v\in S}{Z_v(\xi_v,f_{v,1})}$ is well-defined. Indeed, using \cite{knapp-llc}, Thm.\ 2 and Thm.\ 3 {\it loc. cit.}, it is easy to see that
$$L(s',\Pi_v)=h(s')\cdot \prod_{k=1}^n\Gamma(s'+\mu_{v,k}+n-k+\tfrac12),$$
where $h(s')$ is holomorphic and non-vanishing for all $s'\in\C$. Since $\mu_{v,k}\geq 0$ for all $1\leq k\leq n$, by the self-duality hypotheses, cf.\ \S\ref{sect:highweights}, $L(s',\Pi_v)$ is holomorphic at $s'=\tfrac12$, whence so is $Z_v(\xi^{\epsilon_0}_{v,\underline i, \alpha},f_{v,1})=Z_v(\tfrac12,\xi^{\epsilon_0}_{v,\underline i, \alpha},f_{v,1})$ by Lem.\ \ref{lem:zinfty}. Finally, we let
$$c^t(\Pi_\infty):=(L^S(n,\triv)^2)^{-1} \cdot\sum_{\underline i=(i_1,...,i_{t})}\sum_{\alpha=1}^{\dim E_\mu} s(\underline i)\cdot T_\mu(e_\alpha)\cdot \frac{Z_\infty(\xi^{\epsilon_0}_{\infty,\underline i, \alpha},f_{\infty,1})}{L(n,\triv_\infty)}.$$
Here, both numbers $L(n,\triv_\infty)=\prod_{v\in S_\infty}L(n,\triv_v)=\pi^{-dn/2}\Gamma(\tfrac n2)^d$ and $L^S(n,\triv)$ are non-zero. We claim that Sun's aforementioned result now implies the following

\begin{thm}\label{hyp}
The number $c^t(\Pi_\infty)$ is non-zero.
\end{thm}
\begin{proof}
As a first step and in order to be able to apply Sun's result (\cite{sun}, Thm.\ C), we reduce the problem of showing that $c^t(\Pi_\infty)$ is non-zero to showing that a similarly defined number, $d^t(\Pi_v)$ is non-zero. This latter number will only depend on one archimedean place $v\in S_\infty$, whence we find ourselves back in the setting of \cite{sun}. 

To this end, observe that there is a projection
$$L^t:\Lambda^t(\m_G/\k_{\infty})^*\ira\bigoplus_{a+b=t}\Lambda^a(\g_\infty/\c_\infty)^*\otimes\Lambda^{b}\s^*\twoheadrightarrow \Lambda^r(\g_\infty/\c_\infty)^*\otimes\Lambda^{d-1}\s^*,$$
where $\c_\infty:=\z_\infty\oplus\k_\infty$ and $r=t-d+1$. By reasons of degree, $L^t$ induces an isomorphism of (one-dimensional) vector spaces
$$\mathcal L^t: H^{t}(\m_G,K^\circ_\infty,\Whit^{\psi_\infty}(\Pi_\infty)\otimes E_\mu)[\epsilon_0]\ira H^{r}(\g_\infty,(Z_\infty K_\infty)^\circ,\Whit^{\psi_\infty}(\Pi_\infty)\otimes E_\mu)[\epsilon_0]\otimes \Lambda^{d-1}\s^*_\C,$$
whose effect on the generator $[\Whit^{\psi_\infty}(\Pi_\infty)]^t$ is by sending $X_{\underline i}^*$ to $L^t(X_{\underline i})\in \Lambda^r(\g_\infty/\c_\infty)^*\otimes\Lambda^{d-1}\s^*$. Without loss of generality, we write $L^t(X_{\underline i})=L_r(X_{\underline i})\otimes L_{d-1}(X_{\underline i})$, where $L_r(X_{\underline i})\in\Lambda^r(\g_\infty/\c_\infty)^*$ and $L_{d-1}(X_{\underline i})\in\Lambda^{d-1}\s^*$. Similarly, we factor $\mathcal L^t=\mathcal L_r\otimes\mathcal L_{d-1}$.

As $\z_\infty\subset\h_\infty$, and as moreover $r=\dim_\R \h_\infty/\c_{H,\infty}$, where $\c_{H,\infty}:=\z_\infty\oplus\k_{H,\infty}$, we also have a canonical isomorphism
$$\Lambda^t(\m_H/\k_{H,\infty})^*\ira\Lambda^r(\h_\infty/\c_{H,\infty})^*\otimes\Lambda^{d-1}\s^*.$$
Hence, $L^t$ and $\mathcal L^t$ factor over the injection $\Lambda^t(\m_H/\k_{H,\infty})^*\hra\Lambda^t(\m_G/\k_\infty)^*$ induced by $J$. As a consequence, $c^t(\Pi_\infty)$ is a non-trivial multiple of
$$d^t(\Pi_\infty):=\sum_{\underline i=(i_1,...,i_{t})}\sum_{\alpha=1}^{\dim E_\mu} u(\underline i)\cdot T_\mu(e_\alpha)\cdot \frac{Z_\infty(\xi^{\epsilon_0}_{\infty,\underline i, \alpha},f_{\infty,1})}{L(n,\triv_\infty)},$$
where $u(\underline i)$ is the uniquely defined complex number, such that the restriction of $L_r(X^*_{\underline i})$ to $\Lambda^r(\h_\infty/\c_{H,\infty})^*$ equals $u(\underline i)\cdot L_r(Y_1\wedge..\wedge Y_t)$. %u(i)=s(i)/L_{d-1}(X_i)!
The number $d(^t\Pi_\infty)$ factors as $d^t(\Pi_\infty)=\prod_{v \in S_\infty} d^t(\Pi_v)$, where each local factor $d^t(\Pi_v)$ is defined analogously (using $\Lambda^r(\h_\infty/\c_{H,\infty})^*\cong \bigoplus_{\sum r_v=r}\bigotimes_{v\in S_\infty}\Lambda^{r_v}(\h_v/\c_{H,v})^*$). Therefore, we may finish the proof by showing that $d^t(\Pi_v)$ is non-zero for all $v\in S_\infty$ and we are in the situation considered by Sun in \cite{sun}.

Let $v\in S_\infty$ be an arbitrary archimedean place. For sake of simplicity, we drop the subscript ``$v$'' now everywhere, so, e.g., $\Pi=\Pi_v$, $H=GL_n(\R)\times GL_n(\R)$, $\mu=\mu_v$, $\g=\mathfrak{gl}_{2n}(\R)$ and analogously for all other local archimedean objects. The local integrals $Z(\xi,f_{1})$ define a non-zero homomorphism
$$Z(.,f_1)\in\Hom_H(\Whit^{\psi}(\Pi),\C).$$
This follows from \cite{bumbginzburg}, Thm.\ 2 and Lem.\ \ref{lem:zinfty}. Hence, if we let $\chi:=\triv\times\triv$ be the trivial character of $H$, then $Z(.,f_1)$ can be taken as the map $\varphi_\chi$ in Sun's Thm.\ C, \cite{sun}. Next, recall $T_\mu\in\Hom_{H(\C)}(E_\mu\otimes\C)$ from \S\ref{sect:T}. If we set $w_1:=0=:w_2$, then we may take $T_\mu$ to be the non-zero homomorphism $\varphi_{w_1,w_2}$ from \cite{sun}, Thm.\ C. Hence, {\it loc.\ cit.\ }, Thm.\ C, asserts that the map
$$D:\Hom(\Lambda^{r}\g/\c,\Whit^{\psi}(\Pi)\otimes E_\mu)\longrightarrow \Hom(\Lambda^{r}\h/\c_H,\chi\otimes \C)$$
$$h\mapsto D(h):=(Z(.,f_1)\otimes T_\mu)\circ h\circ \wedge^{r}j_{2n}$$
is non-zero on the one-dimensional sub-space $H^{r}(\g,(ZK)^\circ,\Whit^{\psi}(\Pi)\otimes E_\mu)[\epsilon_0].$ (Here, $j_{2n}$ is Sun's notation for the embedding $\h/\c_H\hookrightarrow\g/\c$.) By the one-dimensionality of the latter cohomology space, it is hence non-zero on $\mathcal L_r([\Whit^{\psi}(\Pi)]^t)$. But, then, $D$ computes
\begin{eqnarray*}
D(\mathcal L_r([\Whit^{\psi}(\Pi)]^t)) & = & (Z(.,f_1)\otimes T_\mu)\circ \mathcal L_r([\Whit^{\psi}(\Pi)]^t)\circ \wedge^{r}j_{2n}\\
& = & \sum_{\underline i=(i_1,...,i_{t})}\sum_{\alpha=1}^{\dim E_\mu} u(\underline i)\,T_\mu(e_\alpha)\, Z(\xi^{\epsilon_0}_{\underline i, \alpha},f_1)\\
& = & L(n,\triv)\cdot d^t(\Pi).
\end{eqnarray*}
Hence, reintroducing the subscript ``$v$'', and recalling that $L(n,\triv_v)=\pi^{-n/2}\Gamma(n/2)\neq 0$, the number $d^t(\Pi_v)$ is non-zero for all archimedean places, whence so is $c^t(\Pi_\infty)$.
\end{proof}

As a consequence of Prop.\ \ref{prop:sigma_poles}, we may hence define the archimedean periods
\begin{equation}\label{eq:arch_period}
p^t({}^\sigma\Pi_\infty):=c^t({}^\sigma\Pi_\infty)^{-1}
\end{equation}
for all $\sigma\in$Aut$(\C)$.

\subsection{Rationality of the residue of the exterior square $L$-function at $s=1$}\label{sect:residue}
This is our first main theorem.

\begin{thm}\label{thm:residue}
Let $F$ be a totally real number field and $G=\GL_{2n}/F$, $n\geq 2$. Let $\Pi$ be a unitary cuspidal automorphic representation of $G(\A)$ (self-dual and with trivial central character), which is cohomological with respect to an irreducible, self-contragredient, algebraic, finite-dimensional representation $E_\mu$ of $G_\infty$. Assume that $\Pi$ satisfies the equivalent conditions of Prop.\ \ref{prop:lift}, i.e., the partial exterior square $L$-function $L^S(s,\Pi,\Lambda^2)$ has a pole at $s=1$. Then, for every $\sigma\in {\rm Aut}(\C)$,
$$
\sigma\left(
\frac{L(\tfrac12,\Pi_f)\cdot {\rm Res}_{s=1}(L^S(s,\Pi,\Lambda^2))}{p^t(\Pi) \ p^t(\Pi_\infty)}\right) \ = \
\frac{L(\tfrac12,{}^\sigma\Pi_f)\cdot {\rm Res}_{s=1}(L^S(s,{}^\sigma\Pi,\Lambda^2))}{p^t({}^\sigma\Pi) \ p^t({}^\sigma\Pi_\infty)}.
$$
In particular,
$$
L(\tfrac12,\Pi_f)\cdot {\rm Res}_{s=1}(L^S(s,\Pi,\Lambda^2)) \ \sim_{\Q(\Pi_f)} \
p^t(\Pi) \ p^t(\Pi_\infty),
$$
where ``$\sim_{\Q(\Pi_f)}$'' means up to multiplication of the right hand side by an element in the number field $\Q(\Pi_f)$.
\end{thm}

\begin{proof}
Let $\Pi$ be as in the statement of the theorem. We consider the commutative diagram \eqref{diagram} in Prop.\ \ref{prop:diagram}: Let
$$\Omega:=\mathcal R\circ T_\mu\circ\mathcal J^t_\mu\circ \Delta^t_\Pi\circ \mathcal F^t_\Pi$$
be the composition of the upper horizontal arrows, and analogously, let ${}^\sigma\Omega$ be the composition of the lower horizontal arrows. Let $\xi_f=\otimes_{v\notin S_\infty}\xi_v\in\Whit^{\psi_f}(\Pi_f)\cong\otimes_{v\notin S_\infty}\Whit^{\psi_v}(\Pi_v)$ be a Whittaker function, such that for $v\notin S$, $\xi_v$ is invariant under $G(\O_v)$ and normalized such that $\xi_v(id_v)=1$. Given $\sigma\in$Aut$(\C)$, let $W^\sigma(\xi_f)={}^\sigma\xi_f\in\Whit^{\psi_f}({}^\sigma\Pi_f)$ be the $\sigma$-twisted Whittaker function, cf.\ \S\ref{sect:Whitt_rational_str}. Then, Prop.\ \ref{prop:diagram} says that
\begin{equation}\label{eq:diagram}
\sigma(\Omega(\xi_f))={}^\sigma\Omega({}^\sigma\xi_f).
\end{equation}
In order to prove the theorem, we make both sides of this equation explicit. To that end, let $[\Pi_\infty]^t$ (resp.\ $[{}^\sigma\Pi_\infty]^t$) the generators of the respective cohomology spaces, \S\ref{sect:generator}. For each $\underline i$ and $\alpha$, let $\varphi_{\underline i,\alpha}:=(\whit^{\psi})^{-1}(\xi^{\epsilon_0}_{\infty,\underline i,\alpha}\otimes\xi_f)\in\Pi$ (resp.\ ${}^\sigma\!\varphi_{\underline i,\alpha}:=(\whit^{\psi})^{-1}({}^\sigma\xi^{\epsilon_0}_{\infty,\underline i,\alpha}\otimes{}^\sigma\xi_f)\in{}^\sigma\Pi$) be the corresponding cuspidal automorphic form. Recall our Schwartz-Bruhat function $\Phi\in\mathscr S(\A^n)$ (resp.\ ${}^\sigma\!\Phi\in\mathscr S(\A^n)$) from \S\ref{sect:SchwartzBruhat} (resp.\ \S\ref{sect:conseq_sigma}), with $m=n$ now. Inserting these functions into Thm.\ \ref{thm:integralrep_extsqr} (likewise, also into \eqref{eq:twist}) and recalling the definition of our archimedean periods $p^t(\Pi_\infty)$ and $p^t({}^\sigma\Pi_\infty)$ from \eqref{eq:arch_period} shows that
equation \eqref{eq:diagram}, induced by our diagram \eqref{diagram}, may be rewritten as
$$\sigma\left(\frac{{\rm vol}_{dg\times dg'}(Z(F)\backslash Z_(\A)/A_G)}{c_n\cdot\hat\Phi(0)}\cdot\frac{L^S(\tfrac12,\Pi)\cdot {\rm Res}_{s=1}(L^S(s,\Pi,\Lambda^2))}{p^t(\Pi) \ p^t(\Pi_\infty)} \prod_{v\in S\backslash S_\infty}\frac{Z_v(\xi_v,f_{v,s})}{L(n,\triv_v)}\right)$$
$$=\frac{{\rm vol}_{dg\times dg'}(Z(F)\backslash Z_(\A)/A_G)}{c_n\cdot{}^\sigma\hat\Phi(0)}\cdot\frac{L^S(\tfrac12,{}^\sigma\Pi)\cdot {\rm Res}_{s=1}(L^S(s,{}^\sigma\Pi,\Lambda^2))}{p^t({}^\sigma\Pi) \ p^t({}^\sigma\Pi_\infty)} \prod_{v\in S\backslash S_\infty}\frac{Z_v({}^\sigma\xi_v,{}^\sigma\!f_{v,s})}{L(n,\triv_v)}.$$
(Recall that $\Pi$ was assumed to have trivial central character.) Invoking our cosmetically tuned choice for ${}^\sigma\!\Phi\in\mathscr S(\A^n)$ from \S\ref{sect:conseq_sigma}, and observing that $L(n,\triv_v)=(1-|\O_v/\wp_v|^{-n})^{-1}\in\Q^\times$ for $v\in S\backslash S_\infty$, this simplifies to
$$\sigma\left(\frac{L^S(\tfrac12,\Pi)\cdot {\rm Res}_{s=1}(L^S(s,\Pi,\Lambda^2))}{p^t(\Pi) \ p^t(\Pi_\infty)} \prod_{v\in S\backslash S_\infty}Z_v(\xi_v,f_{v,s})\right)$$
$$=\frac{L^S(\tfrac12,{}^\sigma\Pi)\cdot {\rm Res}_{s=1}(L^S(s,{}^\sigma\Pi,\Lambda^2))}{p^t({}^\sigma\Pi) \ p^t({}^\sigma\Pi_\infty)} \prod_{v\in S\backslash S_\infty}Z_v({}^\sigma\xi_v,{}^\sigma\!f_{v,s}).$$
Since $\sigma(L(\tfrac12,\Pi_v))=L(\tfrac12,{}^\sigma\Pi_v)\neq 0$ for all $v\in S\backslash S_\infty$, cf.\ \cite{raghuram-imrn}, Prop.\ 3.17, and recalling once more that $S=S(\Pi,\psi)=S({}^\sigma\Pi,\psi)$, we may rewrite this by
$$\sigma\left(\frac{L(\tfrac12,\Pi_f)\cdot {\rm Res}_{s=1}(L^S(s,\Pi,\Lambda^2))}{p^t(\Pi) \ p^t(\Pi_\infty)} \prod_{v\in S\backslash S_\infty}Z_v(\xi_v,f_{v,s})\right)$$
$$=\frac{L(\tfrac12,{}^\sigma\Pi_f)\cdot {\rm Res}_{s=1}(L^S(s,{}^\sigma\Pi,\Lambda^2))}{p^t({}^\sigma\Pi) \ p^t({}^\sigma\Pi_\infty)} \prod_{v\in S\backslash S_\infty}Z_v({}^\sigma\xi_v,{}^\sigma\!f_{v,s}).$$
Observe that $\prod_{v\in S\backslash S_\infty}Z_v(\xi_v,f_{v,s})$ (and $\prod_{v\in S\backslash S_\infty}Z_v({}^\sigma\xi_v,{}^\sigma\!f_{v,s})$) are non-zero. Indeed, if the finite product $\prod_{v\in S\backslash S_\infty}Z_v(\xi_v,f_{v,s})$ were zero, then by the holomorphy of $L^S(s',\Pi)$ at $s'=\tfrac12$ and the holomorphy of the integral $Z_\infty(s',\xi^{\epsilon_0}_{\infty,\underline i, \alpha},f_{\infty,1})$ at $s'=\tfrac12$ (Lem.\ \ref{lem:zinfty} and the paragraph below),
$$\frac{L^S(\tfrac12,\Pi)\cdot L^S(s,\Pi,\Lambda^2)}{L^S(n,\triv)^2}\cdot \prod_{v\in S}\frac{Z_v(\xi_v,f_{v,1})}{L(n,\triv_v)}$$
would have no pole at $s=1$ (Here we let $\xi_v=\xi^{\epsilon_0}_{v,\underline i, \alpha}$ at an archimedean place.). However, reading the proof of Thm.\ \ref{thm:integralrep_extsqr} backwards, respectively, by \cite{bumbginzburg} Thm.\ 1 and Thm.\ 3, the latter expression equals $$\frac{Z(\xi,f_{s})}{L(n,\triv)}$$
as meromorphic functions in $s$. By \cite{bumbginzburg}, Thm.\ 1 and our assumption that $\Pi$ is a functorial lift from SO$(2n+1)$, cf.\ Prop.\ \ref{prop:lift}, the integral $Z(\xi,f_{s})$ has a pole at $s=1$, whereas $L(n,\triv)$ does not by the assumption that $n\geq 2$. Hence, we arrived at a contradiction.

We may therefore finish the proof of the first assertion of Thm.\ \ref{thm:residue} by showing that
$$\sigma\left(\prod_{v\in S\backslash S_\infty}Z_v(\xi_v,f_{v,s})\right)=\prod_{v\in S\backslash S_\infty}Z_v({}^\sigma\xi_v,{}^\sigma\!f_{v,s}).$$
Observing that by a simple change of variable and by our specific choice of $f_{v,s}={}^\sigma\! f_{v,s}$, $Z_v({}^\sigma\xi_v,{}^\sigma\!f_{v,s})=Z_v(\sigma\circ\xi_v,f_{v,s})$, this is achieved by N.\ Matringe in Thm.\ A of the appendix.

The last assertion of the theorem follows by Strong Multiplicity One for cuspidal automorphic representations of $G(\A)$.
\end{proof}

\subsection{Whittaker-Shalika periods and the exterior square $L$-function}\label{sect:WSext}
Theorem \ref{thm:residue} above is accompanied by the following corollary. Recall the non-zero Shalika periods $\omega^\epsilon(\Pi_f)$ from Grobner--Raghuram \cite{grob-ragh}: These were defined by comparing a $\Q(\Pi_f)$-rational structure on a Shalika model of $\Pi_f$ and a $\Q(\Pi_f)$-rational structure on $H^{r}(\g_\infty,(Z_\infty K_\infty)^\circ,\Pi\otimes E_\mu)[\epsilon]$. For details, we refer to \cite{grob-ragh}, Def./Prop.\ 4.2.1. Observe that $\omega^{\epsilon_0}(\Pi_f)$ is well-defined, if we assume that $\Pi$ satisfies the assumptions made in the statement of Thm\ \ref{thm:residue}: Indeed, as these assumptions include that the partial exterior square $L$-function $L^S(s,\Pi,\Lambda^2)$ has a pole at $s=1$, $\Pi$ has a $(\triv,\psi)$-Shalika model by \cite{grob-ragh}, Thm.\ 3.1.1. (The extremely careful reader may also recall Lem.\ \ref{lem:cohom} at this place.) Moreover, by the same reasoning, also the archimedean Shalika period $\omega(\Pi_\infty)=\omega(\Pi_\infty,0)$ from \cite{grob-ragh}, Thm.\ 6.6.2 is
well-defined (and non-zero). Define the Whittaker-Shalika periods
$$P^t(\Pi):=\frac{p^t(\Pi)}{\omega^{\epsilon_0}(\Pi_f)}\quad\quad\textrm{and}\quad\quad P^t(\Pi_\infty):=\frac{p^t(\Pi_\infty)}{\omega(\Pi_\infty)}.$$
Then we have the following result.

\begin{cor}\label{cor:Wext}
Let $\Pi$ be as in the statement of Thm.\ \ref{thm:residue}. Then
$${\rm Res}_{s=1}(L^S(s,\Pi,\Lambda^2)) \ \approx_{\Q(\Pi_f)} \
P^t(\Pi) \ P^t(\Pi_\infty),$$
where ``$\approx_{\Q(\Pi_f)}$'' means up to multiplication of both sides by an element in the number field $\Q(\Pi_f)$.
\end{cor}
\begin{proof}
This is obvious invoking Thm.\ \ref{thm:residue} and \cite{grob-ragh}, Thm.\ 7.1.2.
\end{proof}

\section{A rationality result for the Rankin--Selberg $L$-function}\label{sect:RS}

\subsection{} The content of this section is very closely related to Grobner--Harris--Lapid \cite{grob_harris_lapid}, \S4--\S5 and Balasubramanyam--Raghuram\cite{bala_ragh}, \S 2--\S 3. Indeed, the main result, Thm.\ \ref{thm:Rankin_Selberg}, of this section is Thm.\ 5.3 from \cite{grob_harris_lapid} (but with the totally imaginary field $E$ from \cite{grob_harris_lapid} being replaced by the totally real field $F$ as a ground-field), respectively Thm.\ 3.3.7 from \cite{bala_ragh} (but with the $L$-value $L(1,{\rm Ad}^0,\pi)$ from \cite{bala_ragh} being replaced by the residue of $L^S(s,\Pi\times\Pi^\vee)$ at $s=1$). For the reason of these close analogies we allow ourselves to be rather brief, when it comes to details. Nevertheless, we think it is worthwhile writing down the following, already for reasons of notation, and in order to give precise statements of results in what follows. 

\subsection{Bottom-degree Whittaker periods}\label{sect:bottom_periods}
Let $\Pi$ be as in \S\ref{sect:cusprep1} and let $b:=dn^2$. Then,
$$
\dim_\C H^b(\m_G,K^\circ_\infty,\Pi_\infty\otimes E_\mu)[\epsilon]=1
$$
for all $\epsilon\in \pi_0(G_\infty)^*$. As in \S\ref{sect:cusprep1}, this is a direct consequence of the formula in Clozel \cite{clozel}, Lem.\ 3.14 and the K\"unneth rule. It is hence clear that the entire discussion of \S\ref{sect:rat_struct} and \S\ref{sect:Whitt_rational_str}--\S\ref{sect:periods_sub} carries over to $(\m_G,K^\circ_\infty)$-cohomology in degree $q=b$. In particular, let $\epsilon_1:=((-1)^{n},...,(-1)^{n})\in\pi_0(G_\infty)^*$, i.e., the inverse of the character $\epsilon_0$. The $\Q(\Pi_f)$-structure on $H^b(\m_G,K^\circ_\infty,\Pi\otimes E_\mu)[\epsilon_1]$ imposed by the $\Q(E_\mu)$-structure on $H^b_c(\SSS_G,\E_\mu)$ (cf.\ Rem.\ \ref{rem:rat_structure}) pins down a generator $[\Pi_\infty]^b$ of $H^b(\m_G,K^\circ_\infty,\Pi_\infty\otimes E_\mu)[\epsilon]$,
$$[\Pi_\infty]^b=\sum_{\underline j=(j_1,...,j_{b})}\sum_{\beta=1}^{\dim E_\mu} X^*_{\underline j}\otimes(W^{\psi^{-1}_\infty})^{-1}(\xi^{\epsilon_1}_{\infty,\underline j, \beta})\otimes e_\beta,$$
as in \S\ref{sect:generator}. Observe that here we exchanged the non-trivial additive character $\psi$ by its inverse $\psi^{-1}$ and (for notational clearness only), also the index $\alpha$ by $\beta$. Moreover, in light of Prop.\ \ref{prop:sigma_poles}, for all $\sigma\in$Aut$(\C)$, we obtain non-trivial Whittaker periods $p^b({}^\sigma\Pi)$, unique up to multiplication by elements in $\Q({}^\sigma\Pi_f)^\times$, such that
$$
\xymatrix{
\Whit^{\psi^{-1}_f}(\Pi_f) \ar[rrrr]^{\mathcal F^b_\Pi}\ar[d]_{W^\sigma} & & & &
H^{b}(\m_G,K_{\infty}^\circ, \Pi\otimes E_{\mu})[\epsilon_1]
\ar[d]^{H^{\sigma,b}_\mu} \\
\Whit^{\psi^{-1}_f}({}^\sigma\Pi_f)
\ar[rrrr]^{\mathcal F^b_{{}^\sigma\Pi}}
& & & &
H^{b}(\m_G,K_{\infty}^\circ, {}^\sigma\Pi\otimes {}^\sigma\! E_\mu)[\epsilon_1]
}
$$
commutes. This is the analogue of Prop.\ \ref{prop:periods}, whose proof goes through word for word in the current situation, i.e., for cohomology in degree $b$ instead of $t$. See \cite{raghuram-shahidi-imrn}, Prop./Def. 3.3. In the above diagram, $H^{\sigma,b}_\mu:=(\Delta^{b}_{{}^\sigma\Pi})^{-1}\circ\mathcal H^{\sigma,b}_\mu\circ\Delta^b_\Pi$ is the restriction of $\mathcal H^{\sigma,b}_\mu$ to $H^{b}(\m_G,K_{\infty}^\circ, \Pi\otimes E_{\mu})[\epsilon_1]$, this map being well-defined following by the same argument as in \S\ref{sect:H^sigma}. We leave it to the reader to fill in the remaining details.

\subsection{Another rational diagram}
Recall the Haar measure $dg$ on $G(\A)$ from \S\ref{sect:measures}. In particular, $dg_f$ assigns rational volumes to compact open subgroups of $G(\A_f)$ and $dg_v$ assigns volume $1$ to $K^\circ_v\cong SO(2n)$ at all archimedean places. Notice that $t+b=\dim_\R \SSS_G$. Hence, by the same reasoning as in \S\ref{sect:deRham}, (simply by letting the group $G$ play the role of the group $H$ there) we obtain a surjective linear map
$$\mathcal P:H^{b+t}_c(\SSS_G,\C)\ra\C,$$
induced by the de-Rham-isomorphism. 

Furthermore, by assumption $E_\mu\cong E_\mu^\vee$. So, the canonical pairing $E_\mu\times E_\mu^\vee\ra\C$ induces a pairing $E_\mu\times E_\mu\ra\C$, which we will denote by $\<e_\alpha,e_\beta\>:=e^{\vee}_\beta(e_\alpha)$. The same notation is used for the corresponding induced map on cohomology $H^{q}_c(\SSS_G,\E_\mu\otimes\E_\mu)\ra H^{q}_c(\SSS_G,\C)$. 

In the following proposition, we abbreviate $H^q(\m_G,K^\circ_\infty,\Sigma\otimes E)[\epsilon]$ by $H^q(\Sigma\otimes E)[\epsilon]$, for $\Sigma$ a cuspidal automorphic representation of $G(\A)$ as in \S\ref{sect:cusprep1}, $E$ an algebraic coefficient system as in \S\ref{sect:highweights} and $\epsilon\in\pi_0(G_\infty)^*$.

\begin{prop}\label{prop:diagram2}
The following diagram is commutative:
\begin{equation}\label{diagram2}
\small
\xymatrix{
\Whit^{\psi_f}(\Pi_f)\times \Whit^{\psi^{-1}_f}(\Pi_f) \ar[r]^{W^\sigma\times W^\sigma} \ar[d]^{\mathcal F^t_\Pi\times\mathcal F^b_\Pi}& \Whit^{\psi_f}({}^\sigma\Pi_f)\times  \Whit^{\psi^{-1}_f}({}^\sigma\Pi_f)\ar[d]^{\mathcal F^t_{{}^\sigma\Pi}\times\mathcal F^b_{{}^\sigma\Pi}} \\
H^{t}(\Pi\otimes E_{\mu})[\epsilon_0]\times H^{b}(\Pi\otimes E_{\mu})[\epsilon_1 ]\ar@{^{(}->}[d]^{\Delta^t_\Pi\times\Delta^b_\Pi} \ar[r]^{H^{\sigma,t}_\mu\times H^{\sigma,b}_\mu} & H^{t}({}^\sigma\Pi\otimes {}^\sigma\! E_{\mu})[\epsilon_0]\times H^{b}({}^\sigma\Pi\otimes {}^\sigma\! E_{\mu})[\epsilon_1 ]\ar@{^{(}->}[d]^{\Delta^t_{{}^\sigma\Pi}\times\Delta^b_{{}^\sigma\Pi}}\\
H^t_c(\SSS_G,\E_\mu) \times H^b_c(\SSS_G,\E_\mu) \ar[d]^{\wedge}\ar[r]^{\mathcal H^{\sigma,t}_\mu\times \mathcal H^{\sigma,b}_\mu} & H^t_c(\SSS_G,{}^\sigma\! \E_\mu) \times H^b_c(\SSS_G,{}^\sigma\!\E_\mu) \ar[d]^{\wedge}\\
H^{t+b}_c(\SSS_G,\E_\mu\otimes\E_\mu)\ar[d]^{\<\cdot,\cdot\>} \ar[r]^{\mathcal H^{\sigma,t+b}_{\mu+\mu}} & H^{t+b}_c(\SSS_G,{}^\sigma\!\E_\mu\otimes{}^\sigma\!\E_\mu) \ar[d]^{\<\cdot,\cdot\>}\\
H^{t+b}_c(\SSS_G,\C) \ar[d]^{\mathcal P} \ar[r]^{\mathcal H^{\sigma,t+b}_{0}} & H^{t+b}_c(\SSS_G,\C) \ar[d]^{\mathcal P}\\
\C  \ar[r]^\sigma & \C}
\normalsize
\end{equation}
The vertical maps are linear, whereas the horizontal maps are $\sigma$-linear.
\end{prop}
\begin{proof}
Commutativity of the upper square follows from Prop.\ \ref{prop:periods} and the considerations in \S\ref{sect:bottom_periods}. The square below it is commutative by definition. The commutativity of the other squares is obvious.
\end{proof}

\subsection{An integral representation of the residue of the Rankin--Selberg $L$-function at $s=1$}\label{sect:integralrep_RS}
Recall the Schwartz--Bruhat function $\Phi=\otimes_v\Phi_v\in\mathscr S(\A^{2n})$ from \S\ref{sect:SchwartzBruhat} with $m=2n$ in this case. As it satisfies the conditions of Lem.\ \ref{lem:Eisensteinseries}, the Eisenstein series attached to $\Phi$ has a pole at $s=1$ with constant, non-trivial residue $c_{2n}\hat\Phi(0)$.

Let $U_{2n}$ be the subgroup of upper triangular matrices in $G=\GL_{2n}$, whose diagonal entries are all equal to $1$. For each place $v\notin S_\infty$ let $\xi_v\in\Whit^{\psi_v}(\Pi_v)$ (resp.\ $\xi'_v\in\Whit^{\psi^{-1}_v}(\Pi_v)$) be a local Whittaker function, whereas for each $v\in S_\infty$ we assume that $\xi_v\in\Whit^{\psi_v}(\Pi_v)$ (resp.\ $\xi'_v\in\Whit^{\psi^{-1}_v}(\Pi_v)$) is $SO(2n)$-finite from the right. For such Whittaker functions, the local zeta-integrals
$$\Psi_v(s,\xi_v,\xi'_v,\Phi_v):=\int_{U_{2n}(F_v)\backslash\GL_{2n}(F_v)}\xi_v(g_v) \ \xi'_v(g_v) \Phi_v((0,...,0,1) g_v) |\det(g_v)|_v^s\ dg_v$$
converge for $\Re e(s)\geq 1$, cf.\ \cite{jacshal}, Prop.\ (1.5) and Prop.\ (3.17) \emph{loc. cit}. Assume now that for $v\notin S$, $\xi_v$ and $\xi'_v$ are the unique spherical Whittaker functions which are $1$ on the identity element $id_v\in G(F_v)$ and let $\varphi:=(W^{\psi})^{-1}(\xi)\in\Pi$ and $\varphi':=(W^{\psi^{-1}})^{-1}(\xi')\in\Pi$ be the corresponding cuspidal automorphic forms. 
\begin{thm}\label{thm:integralrep_RS}
Let $\Pi$ be a cuspidal automorphic representation as in \S\ref{sect:cusprep1}. Then
$$c_{2n} \ \hat\Phi(0) \ \cdot\int_{Z(\A)G(F)\backslash G(\A)} \varphi(g) \ \varphi'(g) \ dg = {\rm Res}_{s=1}(L^S(s,\Pi\times\Pi))\cdot\prod_{v\in S}\Psi_v(1,\xi_v,\xi'_v,\Phi_v).$$
\end{thm}
\begin{proof}
This is well-known. Simply combine \cite{jacshal} (5), p.\ 550 with Prop.\ (2.3) in {\it loc.\ cit.}
\end{proof}
By Prop.\ \ref{prop:sigma_poles}, Thm.\ \ref{thm:integralrep_RS} automatically holds for the whole Aut$(\C)$-orbit of $\Pi$. As in \S\ref{sect:conseq_sigma}, and in order to have the notation ready at hand, we would like to make this more precise. Hence, for a given $\sigma\in$Aut$(\C)$, let
$$c_{2n}(\Phi,\sigma):=\sigma\left(\frac{\hat\Phi(0) \ {\rm vol}_{dg}(Z(F)\backslash Z(\A)/A_G)}{c_{2n}}\right)\cdot \frac{c_{2n}}{\hat\Phi(0) \ {\rm vol}_{dg}(Z(F)\backslash Z(\A)/A_G)}.$$
As in \S\ref{sect:conseq_sigma}, this is done purely for cosmetic reasons, and will only be used in the proof of main result of this section, cf.\ Thm.\ \ref{thm:Rankin_Selberg}. Observe that by \S\ref{sect:SchwartzBruhat}, $c_{2n}(\Phi,\sigma)$ is non-zero. In analogy to \S\ref{sect:conseq_sigma}, we let ${}^\sigma\Phi\in\mathscr S(\A^{2n})$ be the Schwartz-Bruhat function, which is defined as follows: At $v\notin S_\infty$, $({}^\sigma\Phi)_v:=\Phi_v$, whereas $({}^\sigma\Phi)_\infty:= c_{2n}(\Phi,\sigma)^{-1}\cdot\Phi_\infty$. Then, Thm.\ \ref{thm:integralrep_RS} in combination with Prop.\ \ref{prop:sigma_poles} shows that

\begin{equation}\label{eq:twist_RS}
c_{2n}{}^\sigma\hat\Phi(0) \int_{Z(\A)G(F)\backslash G(\A)}{}^\sigma\!\varphi(g) \ {}^\sigma\!\varphi'(g) \ dg  ={\rm Res}_{s=1}(L^S(s,{}^\sigma\Pi\times{}^\sigma\Pi))\cdot\prod_{v\in S}\Psi_v(1,{}^\sigma\xi_v,{}^\sigma\xi'_v,{}^\sigma\Phi_v).
\end{equation}

\subsection{Another archimedean period}\label{sect:archimedean_RS}
For $\xi_\infty=\otimes_{v\in S_\infty}\xi_v\in \Whit^{\psi_\infty}(\Pi_\infty)\cong\otimes_{v\in S_\infty} \Whit^{\psi_v}(\Pi_v)$ (resp.\ $\xi'_\infty=\otimes_{v\in S_\infty}\xi'_v\in \Whit^{\psi^{-1}_\infty}(\Pi_\infty)\cong\otimes_{v\in S_\infty} \Whit^{\psi^{-1}_v}(\Pi_v)$) being $K_\infty^\circ$-finite, we abbreviate
$$\Psi_\infty(s,\xi_\infty,\xi'_\infty,\Phi_\infty):=\prod_{v\in S_\infty}\Psi_v(s,\xi_v,\xi'_v,\Phi_v).$$
Recall our generators $[\Pi_\infty]^t$ and $[\Pi_\infty]^b$ from \S\ref{sect:generator} and \S\ref{sect:bottom_periods}. Similar to \S\ref{sect:archimedean}, we let $s(\underline i,\underline j)$ be the unique complex number, such that $X^*_{\underline i}\wedge X^*_{\underline j}=s(\underline i,\underline j)\cdot X_1\wedge...\wedge X_{t+b}$. Putting things together, consider
$$
c(\Pi_\infty):=\sum_{\underline i=(i_1,...,i_{t})}\sum_{\underline j=(j_1,...,j_{b})}\sum_{\alpha=1}^{\dim E_\mu} \sum_{\beta=1}^{\dim E_\mu} s(\underline i,\underline j)\ \<e_\alpha,e_\beta\> \Psi_\infty(1,\xi^{\epsilon_0}_{\infty,\underline i, \alpha}, \xi^{\epsilon_1}_{\infty,\underline j, \beta},\Phi_\infty).
$$
Then there is the following theorem, which follows from Prop.\ 5.0.3 in \cite{bala_ragh}.
\begin{thm}\label{thm:arch_RS}
The number $c(\Pi_\infty)$ is non-zero.
\end{thm}
\begin{proof}
We may adapt the argument given at the beginning of the proof of Thm.\ \ref{hyp}, to see that the non-vanishing of $c(\Pi_\infty)$ may be reduced to showing the non-vanishing of a similarly defined number $d(\Pi_v)$, which only depends on one given archimedean place $v\in S_\infty$. Indeed, there is a projection

$$M_b:\Lambda^b(\m_G/\k_{\infty})^*\ira\bigoplus_{u+v=b}\Lambda^u(\g_\infty/\c_\infty)^*\otimes\Lambda^{v}\s^*\twoheadrightarrow \Lambda^b(\g_\infty/\c_\infty)^*\otimes\Lambda^{0}\s^*,$$
where we wrote again $\c_\infty:=\z_\infty\oplus\k_\infty$. By reasons of degrees of cohomology, $M_b$ induces an isomorphism of (one-dimensional) vector spaces
$$\mathcal M_b: H^{b}(\m_G,K^\circ_\infty,\Whit^{\psi^{-1}_\infty}(\Pi_\infty)\otimes E_\mu)[\epsilon_1]\ira H^{b}(\g_\infty,(Z_\infty K_\infty)^\circ,\Whit^{\psi^{-1}_\infty}(\Pi_\infty)\otimes E_\mu)[\epsilon_1]\otimes \Lambda^{0}\s^*_\C,$$
whose effect on the generator $[\Whit^{\psi^{-1}_\infty}(\Pi_\infty)]^b$ is by mapping $X_{\underline j}^*$ to $M_b(X_{\underline j})\in \Lambda^b(\g_\infty/\c_\infty)^*\otimes\Lambda^{0}\s^*$. Whence, at the cost of re-scaling $M_b(X_{\underline j})$ by the non-trivial factor in $\Lambda^{0}\s^*=\R$, we may and will assume that $M_b(X_{\underline j})\in \Lambda^b(\g_\infty/\c_\infty)^*$. Recall the projection $L^t=L_r\otimes L_{d-1}$ and the isomorphism $\mathcal L^t=\mathcal L_r\otimes \mathcal L_{d-1}$ from the proof of \ref{hyp}\footnote{Observe the difference between the last factors in $L^t$ and $M_b$: While $\Lambda^{0}\s^*=\R$ by convention, the isomorphism $\Lambda^{d-1}\s^*\cong\R$ is not canonical, for which we introduced the notational factor $M_{d-1}$.}. Moreover, observe that there is an isomorphism
$$N^{t+b}:\Lambda^{t+b}(\m_G/\k_{\infty})^*\ira \Lambda^{r+b}(\g_\infty/\c_\infty)^*\otimes\Lambda^{d-1}\s^*,$$
which we factor similarly to $L^t$ as $N^{t+b}(X_1\wedge...\wedge X_{t+b})=N_{r+b}(X_1\wedge...\wedge X_{t+b})\otimes N_{d-1}(X_1\wedge...\wedge X_{t+b})$, where $N_{r+b}(X_1\wedge...\wedge X_{t+b})\in\Lambda^{r+b}(\g_\infty/\c_\infty)^*$ and $N_{d-1}(X_1\wedge...\wedge X_{t+b})\in\Lambda^{d-1}\s^*$. It hence follows that the number $c(\Pi_\infty)$ is a non-trivial multiple of
$$d(\Pi_\infty):=\sum_{\underline i=(i_1,...,i_{t})}\sum_{\underline j=(j_1,...,j_{b})}\sum_{\alpha=1}^{\dim E_\mu} \sum_{\beta=1}^{\dim E_\mu} u(\underline i,\underline j)\ \<e_\alpha,e_\beta\> \Psi_\infty(1,\xi^{\epsilon_0}_{\infty,\underline i, \alpha}, \xi^{\epsilon_1}_{\infty,\underline j, \beta},\Phi_\infty),$$
where $u(\underline i, \underline j)$ is the uniquely defined complex number, such that $L_r(X^*_{\underline i})\wedge M_b(X^*_{\underline j})=u(\underline i,\underline j)\cdot N_{r+b}(X_1\wedge...\wedge X_{t+b})$. The number $d(\Pi_\infty)$ factors as $d(\Pi_\infty)=\prod_{v \in S_\infty} d(\Pi_v)$, where each local factor $d(\Pi_v)$ is defined analogously (using $\Lambda^{r+b}(\g_\infty/\c_{\infty})^*\cong \bigotimes\Lambda^{2n^2+n-1}(\g_v/\c_{v})^*$). Therefore, we may finish the proof by showing that $d(\Pi_v)$ is non-zero for all $v\in S_\infty$. This is the reduction to a single archimedean place $v\in S_\infty$, mentioned at the beginning of the proof. The result hence follows by \cite{bala_ragh}, Prop.\ 5.0.3. 
\end{proof}
In view of the latter non-vanishing result and Prop.\ \ref{prop:sigma_poles}, we may define
\begin{equation}\label{eq:arch_period_RS}
p({}^\sigma\Pi_\infty):=c({}^\sigma\Pi_\infty)^{-1}
\end{equation}
for all $\sigma\in$ Aut$(\C)$. 

\subsection{Rationality of the residue of the Rankin--Selberg $L$-function at $s=1$}\label{sect:Rankin_Selberg}
The main result of this section is the following

\begin{thm}\label{thm:Rankin_Selberg}
Let $F$ be a totally real number field and $G=\GL_{2n}/F$, $n\geq 2$. Let $\Pi$ be a self-dual, unitary, cuspidal automorphic representation of $G(\A)$ (with trivial central character), which is cohomological with respect to an irreducible, self-contragredient, algebraic, finite-dimensional representation $E_\mu$ of $G_\infty$. Then, for every $\sigma\in {\rm Aut}(\C)$,
$$
\sigma\left(
\frac{{\rm Res}_{s=1}(L^S(s,\Pi\times\Pi))}{p^t(\Pi) \ p^b(\Pi) \ p(\Pi_\infty)}\right) \ = \
\frac{{\rm Res}_{s=1}(L^S(s,{}^\sigma\Pi\times{}^\sigma\Pi))}{p^t({}^\sigma\Pi) \ p^b({}^\sigma\Pi) \ p({}^\sigma\Pi_\infty)}.
$$
In particular,
$$
{\rm Res}_{s=1}(L^S(s,\Pi\times\Pi)) \ \sim_{\Q(\Pi_f)^\times} \
p^t(\Pi) \ p^b(\Pi) \ p(\Pi_\infty),
$$
where ``$\sim_{\Q(\Pi_f)^\times}$'' means up to multiplication by a non-trivial element in the number field $\Q(\Pi_f)$.
\end{thm}
\begin{proof}
Let $\Pi$ be as in the statement of the theorem and recall the commutative diagram \eqref{diagram2} from Prop.\ \ref{prop:diagram2}. Let us denote by 
$$\Upsilon:=\mathcal P\circ\<\cdot,\cdot\>\circ\wedge\circ(\Delta^t_\Pi\times\Delta^b_\Pi)\circ(\mathcal F^t_\Pi\times\mathcal F^b_\Pi)$$
the composition of the vertical, linear maps in the left column of \eqref{diagram2}, while we denote by ${}^\sigma\Upsilon$ be the composition of the vertical, linear maps in the right column of \eqref{diagram2}. Let $\xi_f=\otimes_{v\notin S_\infty}\xi_v\in\Whit^{\psi_f}(\Pi_f)\cong\otimes_{v\notin S_\infty}\Whit^{\psi_v}(\Pi_v)$ be a Whittaker function, such that for $v\notin S$, $\xi_v$ is invariant under $G(\O_v)$ and normalized such that $\xi_v(id_v)=1$. Define $\xi'_f:=\overline\xi_f$ to be its complex conjugate function. Then, clearly $\xi'_f=\otimes_{v\notin S_\infty}\xi'_v\in\Whit^{\psi^{-1}_f}(\Pi_f)\cong\otimes_{v\notin S_\infty}\Whit^{\psi^{-1}_v}(\Pi_v)$ is a Whittaker function, where for $v\notin S$, $\xi'_v$ is the unique spherical function in $\Whit^{\psi^{-1}_v}(\Pi_v)$, normalized by $\xi'_v(id_v)=1$. Given $\sigma\in$Aut$(\C)$, let $W^\sigma(\xi_f)={}^\sigma\xi_f\in\Whit^{\psi_f}({}^\sigma\Pi_f)$ (resp.\ $W^\sigma(\xi'_f)={}^\sigma\xi'_f\in\Whit^{\psi^{-1}_f}({}^\sigma\Pi_f)$) be the $\sigma$-twisted Whittaker function, cf.\ \S\ref{sect:Whitt_rational_str}. Then, by Prop.\ \ref{prop:diagram2}
\begin{equation}\label{eq:diagram2}
\sigma(\Upsilon(\xi_f\times\xi'_f))={}^\sigma\Upsilon({}^\sigma\xi_f\times{}^\sigma\xi'_f).
\end{equation}
The first assertion of theorem follows now as in the proof of Thm.\ \ref{thm:residue}, by making both sides of this equation explicit. Therefore, recall our generators  $[\Pi_\infty]^t$ and $[\Pi_\infty]^b$ from \S\ref{sect:generator} and \S\ref{sect:bottom_periods} and let $\varphi_{\underline i,\alpha}:=(\whit^{\psi})^{-1}(\xi^{\epsilon_0}_{\infty,\underline i,\alpha}\otimes\xi_f)\in\Pi$ and $\varphi'_{\underline j,\beta}:=(\whit^{\psi^{-1}})^{-1}(\xi^{\epsilon_1}_{\infty,\underline j,\beta}\otimes\xi'_f)\in\Pi$. Simply by inserting into the definitions of the maps, which make up $\Upsilon$ by composition and using the fact that $\Pi$ was assumed to have trivial central character, we obtain
$$\Upsilon(\xi_f\times\xi'_f)=\frac{{\rm vol}_{dg}(Z(F)\backslash Z(\A)/A_G)}{p^t(\Pi) \ p^b(\Pi)}\cdot\sum_{\underline i,\underline j}\sum_{\alpha,\beta}s(\underline i,\underline j) \ \<e_\alpha,e_\beta\> \ \int_{Z(\A)G(F)\backslash G(\A)} \varphi_{\underline i,\alpha}(g) \ \varphi'_{\underline j,\beta}(g) \ dg.$$
Let $\Phi=\otimes_v\Phi\in\mathscr S(\A^{2n})$ be the global Schwartz-Bruhat function, as chosen in \S\ref{sect:SchwartzBruhat} for $m=2n$. Hence, Thm.\ \ref{thm:integralrep_RS} together with the definition of our archimedean period $p(\Pi_\infty)$, shows that we may rewrite this by
$$\Upsilon(\xi_f\times\xi'_f)=\frac{{\rm vol}_{dg}(Z(F)\backslash Z(\A)/A_G)}{c_{2n}\ \hat\Phi(0)}\cdot \frac{{\rm Res}_{s=1}(L^S(s,\Pi\times\Pi))}{p^t(\Pi) \ p^b(\Pi) \ p(\Pi_\infty)}\cdot\prod_{v\in S\backslash S_\infty}\Psi_v(1,\xi_v,\xi'_v,\Phi_v).$$
Inserting the definition of ${}^\sigma\Phi$, \eqref{eq:diagram2} and \eqref{eq:twist_RS} show that
$$\sigma\left(\frac{{\rm Res}_{s=1}(L^S(s,\Pi\times\Pi))}{p^t(\Pi) \ p^b(\Pi) \ p(\Pi_\infty)}\cdot\prod_{v\in S\backslash S_\infty}\Psi_v(1,\xi_v,\xi'_v,\Phi_v)\right)$$
$$=\frac{{\rm Res}_{s=1}(L^S(s,{}^\sigma\Pi\times{}^\sigma\Pi))}{p^t({}^\sigma\Pi) \ p^b({}^\sigma\Pi) \ p({}^\sigma\Pi_\infty)}\cdot\prod_{v\in S\backslash S_\infty}\Psi_v(1,{}^\sigma\xi_v,{}^\sigma\xi'_v,{}^\sigma\Phi_v).$$
We may therefore finish the proof of the first assertion of the theorem, if we show that
\begin{equation}\label{eq:sigma_RS}
\sigma\left(\prod_{v\in S\backslash S_\infty}\Psi_v(1,\xi_v,\xi'_v,\Phi_v)\right)=\prod_{v\in S\backslash S_\infty}\Psi_v(1,{}^\sigma\xi_v,{}^\sigma\xi'_v,{}^\sigma\Phi_v)\neq 0.
\end{equation}
Observe that we still have the freedom to choose $\xi_v$ and $\xi'_v$ at $v\in S\backslash S_\infty$. If we assume that $\xi_v$ is the unique essential vector in $\Whit^{\psi_v}(\Pi_v)$ (cf.\ \cite{jac-ps-shalika-mathann}, (4.4) and Thm.\ (5.1) {\it loc.\ cit.}) which satisfies $\xi_v(id_v)=1$ and let $\xi'_v:=\overline\xi_v$, then \eqref{eq:sigma_RS} has been proved in \cite{bala_ragh}, 2.3. See also \cite{grob_harris_lapid}, 4.3.

The second assertion of Thm.\ \ref{thm:Rankin_Selberg} follows from the first one, applying Strong Multiplicity One for the cuspidal automorphic spectrum of $G(\A)$ and recalling that ${\rm Res}_{s=1}(L^S(s,\Pi\times\Pi))$ is non-zero. Let us give a short argument for the last claim. In fact, by Thm.\ \ref{thm:arch_RS}, there is a tuple of indices $(\underline i, \underline j, \alpha, \beta)$, such that $\Psi_\infty(1,\xi^{\epsilon_0}_{\infty,\underline i, \alpha}, \xi^{\epsilon_1}_{\infty,\underline j, \beta},\Phi_\infty)$ is non-zero. Abbreviate $\xi_v:=\xi^{\epsilon_0}_{v,\underline i, \alpha}$ and $\xi'_v:=\xi^{\epsilon_1}_{v,\underline j, \beta}$ for $v\in S_\infty$. Hence,
$$\prod_{v\in S}\Psi_v(1,\xi_v,\xi'_v,\Phi_v)=\prod_{v\in S\backslash S_\infty}\Psi_v(1,\xi_v,\xi'_v,\Phi_v)\cdot \Psi_\infty(1,\xi^{\epsilon_0}_{\infty,\underline i, \alpha}, \xi^{\epsilon_1}_{\infty,\underline j, \beta},\Phi_\infty)\neq 0,$$
by \eqref{eq:sigma_RS} and the non-vanishing of $\Psi_\infty(1,\xi^{\epsilon_0}_{\infty,\underline i, \alpha}, \xi^{\epsilon_1}_{\infty,\underline j, \beta},\Phi_\infty)$. So, if ${\rm Res}_{s=1}(L^S(s,\Pi\times\Pi))$ were zero, then, using Thm.\ \ref{thm:integralrep_RS}, the integral
$$\int_{Z(\A)G(F)\backslash G(\A)} \varphi_{\underline i,\alpha}(g) \ \varphi'_{\underline j,\beta}(g) \ dg$$ 
would have to be zero as well. Observe moreover, that without loss of generality, we may had put $\xi^{\epsilon_1}_{\infty,\underline j, \beta}=\overline\xi^{\epsilon_0}_{\infty,\underline i, \alpha}$. Hence, since this latter integral is the Petersson inner product of the non-trivial cusp form $\varphi'_{\underline j,\beta}$ with itself (by our choice of $\xi'_f$ and $\xi^{\epsilon_1}_{\infty,\underline j, \beta}$) this were a contradiction.
\end{proof}

\section{A rationality result for the symmetric square $L$-function}

\subsection{Rationality of the symmetric square $L$-function at $s=1$}
Let $\Pi$ be a cuspidal automorphic representation of $G(\A)$ as in \S\ref{sect:cusprep1} and $\sigma\in$ Aut$(\C)$. Recall the archimedean periods $p^t({}^\sigma\Pi_\infty)$ from \eqref{eq:arch_period} and $p({}^\sigma\Pi_\infty)$ from \eqref{eq:arch_period_RS}. We define our bottom-degree, archimedean period by $$p^b({}^\sigma\Pi_\infty):=\frac{p({}^\sigma\Pi_\infty)}{p^t({}^\sigma\Pi_\infty)}.$$ 
Of course, by Thm.\ \ref{hyp} and Thm.\ \ref{thm:arch_RS}, $p^b({}^\sigma\Pi_\infty)$ is well-defined and non-zero. The following result this is our second main theorem.

\begin{thm}\label{thm:Symm2}
Let $F$ be a totally real number field and $G=\GL_{2n}/F$, $n\geq 2$. Let $\Pi$ be a unitary cuspidal automorphic representation of $G(\A)$ (self-dual and with trivial central character), which is cohomological with respect to an irreducible, self-contragredient, algebraic, finite-dimensional representation $E_\mu$ of $G_\infty$. Assume that $\Pi$ satisfies the equivalent conditions of Prop.\ \ref{prop:lift}, i.e., the partial exterior square $L$-function $L^S(s,\Pi,\Lambda^2)$ has a pole at $s=1$. Then, for every $\sigma\in {\rm Aut}(\C)$,
$$
\sigma\left(
\frac{L(\tfrac12,\Pi_f)\ p^b(\Pi) \ p^b(\Pi_\infty)}{L^S(1,\Pi,{\rm Sym}^2)}\right) \ = \
\frac{L(\tfrac12,{}^\sigma\Pi_f)\ p^b({}^\sigma\Pi) \ p^b({}^\sigma\Pi_\infty)}{L^S(1,{}^\sigma\Pi,{\rm Sym}^2)}.
$$
In particular, 
$$
L^S(1,\Pi,{\rm Sym}^2) \sim_{\Q(\Pi_f)} L(\tfrac12,\Pi_f)\ p^b(\Pi) \ p^b(\Pi_\infty)
$$
where ``$\sim_{\Q(\Pi_f)}$'' means up to multiplication of $L^S(1,\Pi,{\rm Sym}^2)$ by an element in the number field $\Q(\Pi_f)$.
\end{thm}
\begin{proof}
Let $\Pi$ be as in the statement of the theorem. Then $\Pi$ satisfies the assumptions of Thm.\ \ref{thm:residue} and Thm.\ \ref{thm:Rankin_Selberg}. We have the equality
$$L^S(s,\Pi\times\Pi) = L^S(s,\Pi,{\rm Sym}^2) \cdot L^S(s,\Pi,\Lambda^2)$$
as meromorphic functions in $s$, whence, by the assumptions on $\Pi$, we obtain
$${\rm Res}_{s=1}(L^S(s,\Pi\times\Pi)) = L^S(1,\Pi,{\rm Sym}^2) \cdot {\rm Res}_{s=1}(L^S(s,\Pi,\Lambda^2)).$$
Since $L^S(1,\Pi,{\rm Sym}^2)$ is non-zero (cf.\ \cite{shahidi_certainL}, Thm.\ 5.1), the first assertion of the theorem follows by the definition of $p^b(\Pi_\infty)$ and Thm.\ \ref{thm:residue} and Thm.\ \ref{thm:Rankin_Selberg}. The second assertion is now again a consequence of Strong Multiplicity One for the cuspidal automorphic spectrum of $G(\A)$.
\end{proof}

\subsection{Whittaker-Shalika periods and the symmetric square $L$-function}\label{sect:WSsymm}
As in the case of the exterior square $L$-function, we obtain a corollary of our second main theorem, Thm.\ \ref{thm:Symm2}, using the main results of our joint paper \cite{grob-ragh} with Raghuram. Recall the non-zero Shalika periods $\omega^\epsilon(\Pi_f)$ and $\omega(\Pi_\infty)=\omega(\Pi_\infty,0)$ from \S\ref{sect:WSext} above, respectively from \cite{grob-ragh}, Def./Prop.\ 4.2.1 and Thm.\ 6.6.2, therein, their existence being guaranteed as in \S\ref{sect:WSext}.
Define the Whittaker-Shalika periods
$$P^b(\Pi):=p^b(\Pi)\cdot\omega^{\epsilon_0}(\Pi_f)\quad\quad\textrm{and}\quad\quad P^b(\Pi_\infty):=p^b(\Pi_\infty)\cdot\omega(\Pi_\infty).$$
Then, we have the following result.

\begin{cor}\label{cor:Wsymm}
Let $\Pi$ be as in the statement of Thm.\ \ref{thm:Symm2}. Then
$$L^S(1,\Pi,{\rm Sym}^2)\ \approx_{\Q(\Pi_f)} \ P^b(\Pi) \ P^b(\Pi_\infty),$$
where ``$\approx_{\Q(\Pi_f)}$'' means up to multiplication of both sides by an element in the number field $\Q(\Pi_f)$.
\end{cor}
\begin{proof}
This follows directly from Thm.\ \ref{thm:Symm2} and \cite{grob-ragh}, Thm.\ 7.1.2.
\end{proof}

\bigskip

\appendix
\section*{Appendix}\label{appendix}
\begin{center}
{\it by Nadir Matringe}
\end{center}
\vskip 5pt
\subsection*{}
In this appendix, $F$ denotes a non-archimedean local field with valuation $v$ and absolute value $|.|$ (normalised as usual). We will write $|g|$ for $|\det(g)|$ when $g$ is a square matrix, $\O$ for the ring of integers of $F$ and let $\wp=\varpi \O$ be the maximal ideal of $\O$.

\begin{propa*}
Let $\phi\in \sm_c(F)$, $\chi$ a character of $F^\times$, and $m\geq 0$ and integer. Then,
$$T(q^{-s},\chi,m,\phi):=\int_{F^\times} \phi(x)\ \chi(x) \ v(x)^m|x|^s \ d^\times x$$
(with ${\rm vol}_{d^\times x}(\O^\times)=1$) converges for $|q^{-s}|<|\chi(\varpi)|^{-1}$ and can be extended to an element of $L(s,\chi)^m\cdot\C[q^{\pm s}]$. Moreover, if $\sigma\in {\rm Aut}(\C)$, then
$$\sigma(T(q^{-s},\chi,m,\phi))= T(\sigma(q^{-s}),\sigma(\chi),m,\sigma(\phi)).$$
\end{propa*}
\begin{proof}
For $k \in \Z$, we set $$c_k(\chi,\phi):=\int_{\O^\times} \phi(\varpi^k x)\ \chi(x) \ d^\times x,$$ and
$$c(\chi):=\int_{\O^\times } \chi(x)\ d^\times x,$$ so that we have
$c(\chi)=0$, if $\chi$ is ramified, and $c(\chi)=1$, if $\chi$ is unramified.

As $\O^\times $ is compact, take $U$ an open compact subgroup of $\O^\times $ fixing $x\mapsto \phi(\varpi^k x)$ and $\chi$, and let
$\O^\times =\coprod_{i=1}^\ell x_i U$, then
$c_k(\chi,\phi)=\sum_{i=1}^\ell \frac{1}{\ell} \phi(\varpi^k x_i)\chi(x_i)$, hence, as $\sigma(1/\ell)=1/\ell$, we have $$\sigma(c_k(\chi,\phi))=c_k(\sigma(\chi),\sigma(\phi)).$$

Let $a$ be a positive integer such that the support of $\phi$ is contained in $\wp^{a-1}$. Let $b\geq a$ be such that
$\phi$ is constant on $\wp^b$. We have
$$T(q^{-s},\chi,m,\phi)= \sum_{a\leq k \leq b}  c_k(\chi,\phi)k^m\chi(\varpi)^k q^{-ks} +
\sum_{k \geq b} \l c(\chi) k^m \chi(\varpi)^k q^{-ks}.$$
We set
$$A(q^{-s},\chi,m,\phi):=\sum_{a\leq k \leq b} c_k(\chi,\phi) k^m \chi(\varpi)^kq^{-ks}$$
and
$$B(q^{-s},\chi,m,\phi):=\sum_{k \geq b} \l c(\chi) k^m \chi(\varpi)^k q^{-ks}.$$

Suppose that $\chi$ is ramified, i.e., non trivial on $\O^\times $. Then
$$T(q^{-s},\chi,m,\phi)=A(q^{-s},\chi,m,\phi).$$
In this case, we have
\begin{eqnarray*}
\sigma(T(q^{-s},\chi,m,\phi)) & = & \sum_{a\leq k \leq b} \sigma(c_k(\chi,\phi)) k^m \sigma(\chi(\varpi))^k \sigma(q^{-ks})\\
& = & \sum_{a\leq k \leq b} c_k(\sigma(\chi),\sigma(\phi)) k^m \sigma(\chi(\varpi))^k\sigma(q^{-ks})\\
& = & A(\sigma(q^{-s}),\sigma(\chi),m,\sigma(\phi))\\
& = & T(\sigma(q^{-s}),\sigma(\chi),m,\sigma(\phi)),
\end{eqnarray*}
which shows the claim in this case. Suppose now that $\chi$ is unramified. Then there is $P\in \Q[X,X^{-1}]$ (which can be determined explicitly, notice that
the coefficients of $P$ are in $\Q$, hence $\sigma$-invariant) such that
$$B(q^{-s},\chi,m,\phi)= \l \sum_{k \geq b} k^m \chi(\varpi)^k q^{-ks}= \l P(\chi(\varpi)q^{-s})/ (1-\chi(\varpi)q^{-s})^m,$$ which implies that
$$\sigma(B(q^{-s},\chi,m,\phi))=P(\sigma(\chi(\varpi))\sigma(q^{-s}))/(1-\sigma(\chi(\varpi))\sigma(q^{-s}))^m=
B(\sigma(q^{-s}),\sigma(\chi),m,\sigma(\phi)).$$

This implies again $\sigma(T(q^{-s},\chi,\phi))=T(\sigma(q^{-s}),\sigma(\chi),\sigma(\phi))$.

\end{proof}

We denote by $P_n$ the mirabolic subgroup of $G_n=GL(n,F)$, and by $A_n$ the diagonal torus of $G_n$, which is contained in the standard Borel
$B_n$ with unipotent radical $N_n$. For $k\in\{1,\dots,n-1\}$, the group $G_k$ embeds naturally in $G_n$, so the center $Z_k$
of $G_k$ embeds in $A_n$, and $A_n=Z_1\dots Z_n$ (direct product). The following result follows
from Proposition 2.2 of \cite{JPS}. We fix a non-trivial additive character $\psi$ of $F$. If $z_i$ belongs to $Z_i\subset A_n$,
we set $t(z_i)$ to be the element of $F^*$ such that $z_i=diag(t(z_i),I_{n-i})$

\begin{propb*}
Let $\pi$ be an irreducible generic representation of $G_n$, and $\xi\in \Whit^\psi(\pi)$. For each $k \in \{1,\dots,n-1\}$, there exists a finite set $I_k$, a string of characters $(c_{i_k})_{i_k\in I_k}$ of $F^*$, non-negative integers $(m_{i_k}^\xi)_{i_k\in I_k}$, and functions $(\phi_{i_k}^\xi)_{i_k\in I_k}$ such that
$$\xi(z_1\dots z_{n-1})=\sum_{k=1}^{n-1}\sum_{i_k\in I_k} \prod_{k=1}^{n-1} c_{i_k}(t(z_k)) \ v(t(z_k))^{m_{i_k}^W} \ \phi_{i_k}^\xi(t(z_k)).$$
(The characters $c_{ik}$, which we allow to be equal, depend only on $\pi$.)
\end{propb*}

We denote by $w_n$ the element of the symmetric group $\mathfrak{S}_n$ naturally embedded in $G_n$, defined by
$$\begin{pmatrix} 1 & 2 & \dots & m-1 & m & m+1 & m+2 & \dots & 2m-1 & 2m \\
 1 & 3 & \dots & 2m-3 & 2m-1 & 2 & 4 & \dots & 2m-2 & 2m
 \end{pmatrix}$$ when $n=2m$ is even, and by
$$\begin{pmatrix} 1 & 2 & \dots & m-1 & m & m+1 & m+2 &\dots & 2m & 2m+1 \\
 1 & 3 & \dots & 2m-3 & 2m-1 & 2m+1 & 2 & \dots & 2m-2 & 2m
 \end{pmatrix}$$ when $n=2m+1$ is odd. We denote by $L_n$ the standard Levi subgroup of $G_n$
which is $G_{\lfloor (n+1)/2 \rfloor}\times G_{\lfloor n/2 \rfloor}$ embedded by the map $(g_1,g_2)\mapsto diag(g_1,g_2)$.
We denote by $H_n$ the group $L_n^{w_n}=w_n^{-1} L_n w_n$, by $J(g_1,g_2)$ the matrix
$w_n^{-1} diag(g_1,g_2)w_n$ of $H_n$ (with $diag(g_1,g_2)\in L_n$).
Let $r$ be a positive integer. Thanks to the Iwasawa decomposition $G_r=N_r\cdot A_r \cdot G_r(\O)$, if $\chi$ is an unramified character of $A_r$, then the map
$$\tilde{\chi}:n\cdot a\cdot k\mapsto \chi(a)$$
is well defined on $G_r$. For example, if $\delta_r$ is the modulus character of the maximal
parabolic subgroup of type $(r-1,1)$ restricted to $A_r$, we have a map $\tilde{\delta_r}$ on $G_r$. Similarly, if
$$\lambda:z_1\dots z_r \in A_r \mapsto |t(z_1) \dots t(z_{r-1})|,$$
the map $\tilde{\lambda}$ is also defined on $G_r$, and left invariant under $Z_r$.

\begin{thma*}
Let $\pi$ be an irreducible generic representation of $G_n$ with trivial central character and $\xi\in \Whit^\psi(\pi)$. Set $m'={\lfloor (n+1)/2\rfloor}$ and $m={\lfloor n/2\rfloor}$. The integral
$$Z(\xi,q^{-s}):=\int_{N_{m}\backslash G_{m}}\int_{Z_{m'} N_{m'}\backslash G_{m'}} \xi(J(h,g))\ \tilde{\delta}_{m'}(g)\ |g|^s \ \tilde{\lambda}(h)^s \ dg \ dh$$ (with the normalisations
$dg=d^\times a \ dk$ with $d^\times a(A_m(\O))=1$ and $dk(G_{m}(\O))=1$, $dh=d^\times b \ dk$ with $d^\times b(Z_{m'}(\O)\backslash A_{m'}(\O))=1$ and $dk(G_{m'}(\O))=1$) converges absolutely for $|q^{-s}|$ small enough. It extends to an element of $\C(q^{-s})$, which satisfies that for all $\sigma\in {\rm Aut}(\C)$, one has $$\sigma(Z(\xi,q^{-s}))=Z(\sigma\circ\xi, \sigma(q^{-s})).$$
\end{thma*}
\begin{proof}
Let $\delta'$ be the modulus character of $B_m$ and $\delta''$ that of $B_{m'}$. Let $U'\times U$ be a compact open subgroup of
$G_{m'}(\O)\times G_m(\O)$ such that $J(U'\times U)$ fixes $\xi$ on the right, and write $G_{m'}(\O)\times G_n(\O)=\coprod_{j=1}^\ell x_j U'\times y_j U$. We have
$$Z(\xi,q^{-s})=\sum_{j=1}^\ell \frac{1}{\ell}\int_{A_m}\int_{A_{m'-1}} \xi_j(J(b,a)) \ \delta_{m'}(a)\ \delta'(a)^{-1} \ \delta''(b)^{-1}|a|^s|b|^s \ d^\times a \ d^\times b,$$
where
$\xi_j(g)=\xi(g J(x_j,y_j))$.
We identify $A_m\times A_{m'-1}$ with $A_{n-1}$ by $(a,b)\mapsto J(b,a)$, and set
$\chi$ the character of $A_{n-1}$ defined by $J(b,a)\mapsto \delta_{m'}(a)\delta'(a)^{-1}\delta''(b)^{-1}$. The previous integral
becomes
$$Z(\xi,q^{-s})=\sum_{j=1}^\ell \frac{1}{\ell}\int_{A_{n-1}} \xi_j(z)\ \chi(z)\ |z|^s \ d^\times z,$$
We set $\chi_k$ to be the restriction of $\chi$ to $Z_k$. It takes values in $q^{\Z}\subset \Q$. If we now apply the second proposition of this appendix, we obtain that
$$Z(\xi,q^{-s})$$ is the sum for $k$ between $1$ and $n-1$, $i\in I_k$, and $j\in \{1,...,\ell\}$ of
$$\frac{1}{\ell}\prod_{k=1}^{n-1}T((q^{-s})^k,c_{i_k}\chi_k,m_{i_k}^{\xi_j},\phi_{i_k}^{\xi_j}).$$ This implies,
according to the first proposition of this appendix, that
$$\sigma(Z(\xi,q^{-s}))$$ is the sum for $k$ between $1$ and $n-1$, $i\in I_k$, and $j\in \{1,...,\ell\}$ of
$$\frac{1}{\ell}\prod_{k=1}^{n-1}T((\sigma(q^{-s}))^k,\sigma(c_{i_k})\chi_k,m_{i_k}^{\xi_j},\sigma(\phi_{i_k}^{\xi_j})).$$
This means that $\sigma(Z(\xi,q^{-s}))$ is equal to $Z(\sigma\circ\xi,\sigma(q^{-s}))$.
\end{proof}

\vskip 10pt
\footnotesize
{\sc Nadir Matringe: Universit\'e de Poitiers, Laboratoire de Math\'ematiques et Applications,
T\'el\'eport 2 - BP 30179, Boulevard Marie et Pierre Curie, 86962, Futuroscope Chasseneuil Cedex. Email: }
\\ {\it E-mail address:} {\tt Nadir.Matringe@math.univ-poitiers.fr}

\end{document}